\documentclass[a4paper,10pt,leqno]{amsart}
        \title{An additivity theorem for cobordism categories}

       \author{Wolfgang Steimle}
\address{Institut f\"ur Mathematik\\
	Universit\"at Augsburg\\
        D-86135 Augsburg, Germany}
\email{wolfgang.steimle@math.uni-augsburg.de}
\date{\today}
 
\usepackage[active]{srcltx}
\usepackage{pdfsync}
\usepackage{hyperref}
\usepackage{calc}
\usepackage{enumerate,amssymb, amsmath, amsthm, mathrsfs}
\usepackage{graphicx}
\usepackage[arrow,curve,matrix,tips,2cell]{xy}
  \SelectTips{eu}{10} \UseTips
  \UseAllTwocells

\usepackage{tikz}
\usetikzlibrary{decorations.pathmorphing}

\DeclareMathAlphabet{\matheurm}{U}{eur}{m}{n}
\newcommand{\Top}{\matheurm{Top}}

\newcommand{\Cat}{\matheurm{Cat}}

\newcommand{\arincl}{\ar@{^{(}->}}
\newcommand{\arinclinv}{\ar@{_{(}->}}

\newcommand{\eps}{\varepsilon}

\newcommand{\op}{^{op}}

\newcommand{\power}[1]{\mathcal{P}(#1)}
\newcommand{\proj}{\operatorname{proj}}
\newcommand{\inv}{^{-1}}

\newcommand{\dell}{\partial}

\newcommand{\cob}{\mathrm{cob}}

\newcommand{\cart}[1]{{#1}\mathrm{-cart}}
\newcommand{\colbun}[1]{\cor{#1}\text{-}\matheurm{Bun}^\mathrm{col}}

\DeclareMathOperator{\Aut}{Aut}

\DeclareMathOperator*{\colim}{colim}

\DeclareMathOperator{\End}{End}

\DeclareMathOperator{\id}{id}

\DeclareMathOperator{\map}{map}
\DeclareMathOperator{\mor}{mor}

\DeclareMathOperator{\ob}{ob}

  \newcommand{\IR}{\mathbb{R}}

  \newcommand{\calc}{\mathcal{C}}
  \newcommand{\cald}{\mathcal{D}}

  \newcommand{\bfB}{{\mathbf B}}

  \newcommand{\bfR}{{\mathbf R}}
  \newcommand{\bfS}{{\mathbf S}}

\newcommand{\cor}[1]{\langle #1 \rangle}
\newcommand{\Emb}{\operatorname{Emb}}
\newcommand{\Diff}{\operatorname{Diff}}

\newcommand{\rc}[2]{\IR^{#1}_{\cor{#2}}}

\newcounter{commentcounter}

\theoremstyle{plain}
\newtheorem{theorem}{Theorem}[section]

\newtheorem{lemma}[theorem]{Lemma}
\newtheorem{corollary}[theorem]{Corollary}

\theoremstyle{definition}
\newtheorem{definition}[theorem]{Definition}

\theoremstyle{remark}

\newtheorem*{remark*}{Remark}
\newtheorem*{remarks*}{Remarks}
\newtheorem*{examples*}{Examples}
\newtheorem{example}[theorem]{Example}

\newtheorem{remark}[theorem]{Remark}

\makeatletter\let\c@equation=\c@theorem\makeatother

\begin{document}

\begin{abstract}
Using methods inspired from algebraic $K$-theory, we give a new proof of the Genauer fibration sequence, relating the cobordism categories of closed manifolds with cobordism categories of manifolds with  boundaries, and of the B\"okstedt-Madsen delooping of the cobordism category. Unlike the existing proofs, this approach generalizes to other cobordism-like categories of interest. Indeed we argue that the Genauer fibration sequence is an analogue, in the setting of cobordism categories, of Waldhausen's Additivity theorem in algebraic $K$-theory.
\end{abstract}

\maketitle

\section{Introduction}

There are various reasons to investigate cobordism-like categories of mathematical objects which cannot be decomposed using transversality. The following examples are under current investigation:
\begin{enumerate}
 \item The $h$-cobordism category \cite{RS_Hcob},
 \item Cobordism categories of manifolds equipped with a metric of positive scalar curvature \cite{ERW_psc},
 \item The cobordism category associated to a Waldhausen category \cite{RS_wald}, whose homotopy type has been shown to agree, up to degree shift, with the algebraic $K$-theory space from \cite{Waldhausen(1985)}.
 \item Cobordism categories of Poincar\'e chain complexes in the sense of Ranicki, and more generally of stable $\infty$-categories equipped with a non-degenerate quadratic functor \cite{CDHHLMNNS}, which will be shown to give rise to real algebraic $K$-theory. 
\end{enumerate}

The homotopy types of these cobordism categories are not known to be described by Thom spectra, and can hardly, if at all, be computed. For this reason, it is important to have theorems available that at least \emph{compare} the homotopy types for various inputs. This program has been successfully implemented in algebraic $K$-theory, where Waldhausen's additivity theorem \cite{Waldhausen(1985)} is probably the most fundamental theorem of this kind.

In this paper we formulate and prove a ``local-to-global principle'' designed to produce homotopy fiber sequences of cobordism-like categories, and hence long exact sequences on their homotopy groups. Roughly, it states that under certain assumptions, pull-back squares
\[\xymatrix{
 \cald'\times_\cald \calc \ar[rr] \ar[d] && \calc\ar[d]^P\\
 \cald' \ar[rr]^F && \cald
}\]
of (not necessarily unital, but  ``weakly unital'') categories induce, by geometric realization, homotopy pull-back squares of spaces; the key assumption here is that morphisms in the  target of $P$ can be universally lifted through $P$  in a specific sense, formalized by the  notion of cartesian and cocartesian fibration. 

Then we apply this local-to-global principle in the setting  of classical cobordism categories $\cob_d$ of \cite{GMTW} and its version with boundaries $\cob_{d,\dell}$ of \cite{Genauer} (actually, possibly with higher corners and tangential structures, see section \ref{sec:genauer}). It will turn out that in this case we recover the following theorem, due to Genauer \cite{Genauer}:

\begin{theorem}\label{thm:main_genauer}
 There is a homotopy fibration sequence
\begin{equation}\label{eq:genauer_sequence}
B\cob_{d}\to B\cob_{d,\dell} \xrightarrow{B\dell} B\cob_{d-1},
\end{equation}
where the functor $\dell$ is given by taking boundaries.
\end{theorem}

In his proof of this result, Genauer first determined the homotopy type of the middle term of \eqref{eq:genauer_sequence} in analogy to the computation for the left and right hand term of \cite{GMTW}, and Theorem \ref{thm:main_genauer} follows from the fibration sequence
\[
 MT(d) \to \Sigma^\infty_+(BO(d)) \to MT(d-1)
\]
of Thom spectra from \cite[(3.3)]{GMTW}. In contrast, our method does not yield, nor require, knowledge of the homotopy types of any of the terms in the fibration sequence. 

In Waldhausen's approach  to algebraic $K$-theory, the Additivity  theorem implies that the $K$-theory space can be delooped by iterating the $S_\bullet$-construction. Similarly, we formally deduce from the local-to-global principle  the following result, which is due to B\"okstedt-Madsen \cite{BM} (see section \ref{sec:delooping} for more details):

\begin{theorem}\label{thm:main_BM}
$B\cob_d$ has a (usually non-connective) delooping whose $j$-th term is the classifying space of the $d$-dimensional $(j+1)$-tuple cobordism category. 
\end{theorem}

A virtue of the local-to-global principle is that it has interesting consequences in various situations. As will be shown in \cite{CDHHLMNNS}, it leads to an additivity theorem for real $K$-theory in the setting of stable $\infty$-categories; it will also be shown in that paper that we recover, from the local-to-global principle, Waldhausen's Additivity theorem for the algebraic $K$-theory of stable infinity categories. 

In joint work with George Raptis we will use the local-to-global principle to continue our study of the homotopy type of the $h$-cobordism category and its relation to algebraic $K$-theory, initiated in \cite{RS_Hcob}. It also applies in the context of (ii) in which case it compares the cobordism category of p.s.c.\ metrics with the usual one, called ``Fibre theorem'' in \cite{ERW_psc}, see Remark 4.1.6. in \emph{loc.\ cit.}. It also leads to a proof of a version of Genauer's theorem in the context of topological cobordism categories, as studied by Gomez-Lopez--Kupers (see \cite{GLK}) and gives a very short route to the homotopy type of the cobordism category with Baas-Sullivan singularities, discussed by Perlmutter \cite[Theorem 1.1]{Perlmutter}. 

I am grateful to Thomas Nikolaus for pointing out to me a generalization of the local-to-global principle, in the setting of simplicial sets, that  implies Quillen's Theorem B (see Theorem \ref{thm:cocartesian_Quillen_B}). Recently, Steinebrunner \cite{Steinebrunner} has proven a generalization of the local-to-global principle and used it to discuss the homotopy type of the cospan category of finite sets.

\section{The local-to-global principle}\label{sec:local_to_global}

Let us now formulate the local-to-global principle, first in the setting of topological categories. We emphasize that the categories under consideration are not assumed to be unital. We will deduce this local-to-global principle from a corresponding statement for unital categories (actually, simplicial sets, see Theorem \ref{thm:main_quasicat}) through a process of finding degeneracies. We view this generalization to the non-unital setting as the main contribution of this part of the paper. 

Let us recall that a (small, non-unital) topological category $\calc$ consists of the datum of two topological spaces $\ob(\calc)$ and $\mor(\calc)$, two  continuous maps $s,t\colon \mor(\calc)\to \ob(\calc)$
(called \emph{source map} and \emph{target map}), and a continuous map $\circ\colon \mor(\calc)\times_{\ob(\calc)} \mor(\calc)\to \mor(\calc)$ over $\ob(\calc)\times \ob(\calc)$ (called \emph{composition map}) which satisfies the usual associativity law. A \emph{continuous functor} $\calc\to \cald$ between topological categories consists of two continuous maps $\ob(\calc)\to \ob(\cald)$ and $\mor(\calc)\to \mor(\cald)$, compatible with source, target, and composition.

For a topological category $\calc$ and two objects $c,d\in \ob(\calc)$, we write, as usual, $f\colon c\to d$ for a morphism $f\in \mor(\calc)$ such that $s(f)=c$ and $t(f)=d$; and we denote by $\calc(c,d)\subset\mor(\calc)$ the subspace of morphisms with source $c$ and target $d$.

A morphism $f\colon c\to d$ is called an \emph{equivalence} if for any $t\in \ob\calc$, the maps
\[f\circ - \colon \calc(t,c) \to \calc(t,d) \quad \mathrm{and} \quad -\circ f\colon \calc(d,t)\to \calc(c,t)\]
are weak homotopy  equivalences. 
%A morphism $u\colon c\to c$ is called \emph{weak unit} if it is an equivalence and if $u\circ u\simeq u$ in $\calc(c,c)$. 
A topological category is \emph{weakly unital} if any object is the target of an equivalence. (This is equivalent to the condition that any object $c$ allows a \emph{weak unit}, that is, an equivalence $u\colon c\to c$ such that $u\circ u\simeq u$ in $\calc(c,c)$, compare the proof Lemma \ref{lem:translation_Segal_quasicat} (iv).)
%A \emph{continuous functor} $F\colon \calc\to\cald$ between categories consists of two continuous maps $\ob\calc\to \ob\cald$ and $\mor \calc\to\mor\cald$ which are compatible with $s$, $t$, and $\circ$. If $\calc$ and $\cald$ are weakly unital, then $F$ is called \emph{weakly unital} provided $F$ sends weak units to weak units. 
A continuous functor $F\colon \calc\to\cald$ between weakly unital topological categories is called \emph{weakly unital} provided $F$ preserves equivalences (equivalently, sends weak units to weak units). A unital category is weakly unital, because then the identity map on an object is a weak unit; similarly a unital functor is weakly unital, because any weak unit of $c$ is homotopic to $\id_c$ in $\calc(c,c)$.

The \emph{nerve} of a topological category $\calc$ is the semi-simplicial space $N_\bullet \calc$ (no degeneracies) whose $n$-simplices are the $n$-strings of composable morphisms in $\calc$. Its classifying space is the \emph{geometric realization} $B\calc$ of $\calc$.

\begin{definition}
Let $P\colon \calc\to \cald$ be a continuous functor. 
\begin{enumerate}
\item $P$ is a \emph{level fibration} if it is a (Serre) fibration on each level of the nerve.
\item $P$ is a \emph{local fibration} if both maps
\begin{align*}
 (P,s,t)\colon  \mor \calc &\to \mor\cald\times_{\ob\cald\times\ob\cald} (\ob\calc\times \ob\calc), \\
 P\colon \ob\calc&\to \ob\cald
\end{align*}
are fibrations. $\calc$ is \emph{locally fibrant} if the unique functor $\calc\to *$ is a local fibration, that is, if the source-target map
\[
(s,t)\colon \mor(\calc)\to \ob(\calc)\times \ob(\calc)
\]
is a fibration. 
 \item A morphism $f\colon c\to c'$ in $\calc$ is \emph{$P$-cartesian} if for all $t\in \ob\calc$, the following commutative square is a homotopy pull-back:
\begin{equation}\label{eq:cartesian_topological}
 \xymatrix{
\calc(t, c) \ar[r]^{f\circ -} \ar[d]^P & \calc(t, c') \ar[d]^P\\
\cald(P(t), P(c)) \ar[r]^{P(f)\circ -} & \cald(P(t), P(c')),
}
\end{equation}
\item $P$ is a \emph{cartesian fibration} if for any morphism $g\colon d\to d'$ and any $c'\in \ob\calc$ such that $P(c')=d'$, there is a $P$-cartesian morphism $f\colon c\to c'$ such that $P(f)=g$.
\item $P$ is a \emph{cocartesian fibration} if $P\op\colon \calc\op\to \cald\op$ is a cartesian fibration.
\end{enumerate}
\end{definition}

\begin{remarks*}
\begin{enumerate}
 \item A local fibration is a level fibration. 
 \item If $P$ is a  local fibration, then all induced maps $\calc(c,c') \to \cald(P(c), P(c'))$ between morphism spaces are fibrations.
%\item The classes of level fibrations, local fibrations and of (co-)cartesian fibrations are stable under composition. Level and local fibrations are stable under taking pullbacks. (Co-)cartesian fibrations are often stable under pullbacks (see Lemma \ref{lem:cartesian_preserved_under_pull_back}). 
\item If $\calc$ and $\cald$ are discrete and unital categories, then the above notion of (co-)cartesian fibration agrees with the usual notion of (co-)cartesian (or Grothendieck) fibration.
\end{enumerate}
\end{remarks*}

\begin{theorem}[Local-to-global principle]\label{thm:main_categorical}
Let 
\[\cald'\xrightarrow{F} \cald \xleftarrow{P} \calc\]
be a diagram of weakly unital, locally fibrant topological categories and weakly unital continuous functors. If  $P$ a level, cartesian and cocartesian fibration, then the pull-back diagram
\begin{equation*}%\label{eq:main_diagram}
\xymatrix{
\cald'\times_\cald \calc \ar[r] \ar[d]^{P'} & \calc \ar[d]^P
\\
\cald' \ar[r]^F & \cald
}
\end{equation*}
induces a homotopy pull-back diagram 
\begin{equation*}%\label{eq:main_diagram}
\xymatrix{
B(\cald'\times_\cald \calc) \ar[r] \ar[d]^{BP'} & B\calc \ar[d]^{BP}
\\
B\cald' \ar[r]^{BF} & B\cald
}
\end{equation*}
of classifying spaces.
\end{theorem}

Here $\cald'\times_\cald \calc$ denotes the pull-back in the category of topological categories, explicitly given by the pull-back of spaces of objects and morphisms.

\begin{remark}
The condition on weak unitality cannot be dropped, as can be seen by taking $P$ the functor $\dell\colon \cob_{d, \dell} \to \cob_{d-1}$ in Genauer's sequence (see section \ref{sec:genauer}), and $F$ the inclusion of the subcategory consisting of only the object $\emptyset$ and no morphisms.
\end{remark}

Generally, it requires some care to make the composition operation in the cobordism category strictly associative, and for many perspectives it is easier and more natural to view the cobordism category as a semi-Segal space.

\begin{definition}
A \emph{semi-Segal space} is a semi-simplicial space  $\calc$ that satisfies the \emph{Segal condition} that for each $n$, the Segal map
\[\calc_n \to \calc_1 \times^h_{\calc_0} \dots \times^h_{\calc_0} \calc_1\]
into the $n$-fold iterated homotopy pull-back is  a weak homotopy equivalence.
\end{definition}

A map of semi-Segal spaces is, by definition, a map of the underlying semi-simplicial spaces. 
%Notice that a semi-Segal space is not necessarily unital, in the sense that we do \emph{not} require the existence of degeneracy maps. A simplicial space whose underlying semi-simplicial space is a semi-Segal space will be called \emph{Segal space}. 

We identify a topological category $\calc$ with a semi-simplicial space, via the nerve. With this convention, a locally fibrant topological category is a semi-Segal space, with the Segal map being a homeomorphism
into the actual pull-back.

Next we define a notion of $P$-(co-)cartesian morphism in a semi-Segal space $\calc$. For $\sigma\in \calc_k$ and $n\geq k$, we denote by $\calc_n/\sigma$ the homotopy fiber, over $\sigma$, of the map $\calc_n\to \calc_k$ sending an $n$-simplex to its last $k$-simplex. With this notation, we define:

\begin{definition}
Let $P\colon \calc\to \cald$ be a map of semi-Segal spaces.
\begin{enumerate}
 \item  A 1-simplex $f\colon c\to d$ in $\calc$ (\emph{i.e.,} $f\in \calc_1$ with $d_1(f)=c$ and $d_0(f)=d$) is called \emph{$P$-cartesian} if the square
\begin{equation}\label{eq:cartesian_semisimplicial}
 \xymatrix{
 \calc_2/f \ar[r]^{d_1} \ar[d]^P & \calc_1/d \ar[d]^P\\
 \cald_2/P(f) \ar[r]^{d_1} & \cald_1/P(d)
}
\end{equation}
is a homotopy pull-back square. We denote by $\calc_1^{\cart{P}}\subset \calc_1$ the subset of $P$-cartesian morphisms.
\item $P$ is called \emph{cartesian fibration} if the map
\[
 (P, d_0)\colon \calc_1^{\cart{P}} \to \cald_1 \times^h_{\cald_0} \calc_0
\]
into the homotopy pull-back is $\pi_0$-surjective (that is, surjective on the level of path components). $P$ is called \emph{cocartesian fibration} if $P\op\colon \calc\op\to \cald\op$ is a cartesian fibration.
\end{enumerate}
\end{definition}

\begin{remark}\label{rem:cartesian_topological_vs_Segal}
Let $P\colon \calc\to \cald$ be a map of locally fibrant topological categories.  
\begin{enumerate}
 \item A morphism in $\calc$ is $P$-cartesian in the sense of topological categories if and only if it is $P$-cartesian in the sense of semi-Segal spaces.
 \item If $P$ is a cartesian fibration in the sense of topological categories, then it is a cartesian fibration in the sense of semi-Segal spaces.
 \item The converse of (ii) holds if $P$ is a local fibration.
\end{enumerate}
\end{remark}

\begin{proof}
The space $\calc_1/d$ is equivalent to the space of all morphisms in $\calc$ ending at $d$; also $\calc_2/f$ is equivalent to the actual fiber of $\calc_2\to \calc_1$ over $f$, which agrees with the space of morphisms ending at $c$. Under this equivalence the upper horizontal map in \eqref{eq:cartesian_semisimplicial} corresponds to the map given by composition with $f$. Similarly, the lower horizontal map is equivalent to composition with $P(f)$. Then the square \eqref{eq:cartesian_topological} is obtained from \eqref{eq:cartesian_semisimplicial} by taking homotopy fibers, over $t\in \calc_0$, along the first-vertex map. So \eqref{eq:cartesian_semisimplicial} is homotopy cartesian if and only if \eqref{eq:cartesian_topological} is homotopy cartesian for all $t\in \calc_0$.  This shows (i). 

Now, if $P$ is a cartesian fibration in the sense of topological categories, then the map
\[(P, d_0)\colon \calc_1^{\cart P} \to \cald_1\times_{\cald_0} \calc_0\]
is surjective, hence $\pi_0$-surjective, and the pull-back on the right is a homotopy pull-back, by local fibrancy of $\cald$. This shows (ii).

Finally, we note that $\calc_1^{\cart P}\subset \calc_1$ consists of entire path components. Therefore, if $P$ is a local fibration, the map $(P, d_0)$ in question is a fibration, and $\pi_0$-surjectivity implies actual surjectivity. This shows (iii).
\end{proof}

\begin{definition}
Let $\calc$ be a semi-Segal space and $f\in \calc_1$. Then $f$ is called \emph{equivalence} if it is both $P$-cartesian and $P$-cocartesian, for the unique map $P\colon \calc\to *$ to the terminal object. We denote by $\calc_1^\simeq\subset \calc_1$ the subspace of equivalences. $\calc$ is called \emph{weakly unital} if the map
\[d_0\colon \calc_1^\simeq\to \calc_0\]
is $\pi_0$-surjective. A map of weakly unital semi-Segal spaces $F\colon \calc\to \cald$ is called \emph{weakly unital} if it preserves equivalences.
\end{definition}

\begin{remark}
Let $P\colon \calc\to \cald$ be a map of locally fibrant topological categories. 
\begin{enumerate}
 \item A morphism in $\calc$ is an equivalence in the sense of topological categories if and only if it is an equivalence in the sense of semi-Segal spaces.
 \item $\calc$ (resp., $P$) is weakly unital in the sense of topological categories if and only if it is weakly unital in the sense of semi-Segal spaces.
\end{enumerate}
\end{remark}

\begin{proof}
(i) is a special case of Remark \ref{rem:cartesian_topological_vs_Segal}, (i). Also, 
\[d_0\colon \calc_1^\simeq \to \calc_0\]
is a fibration by local fibrancy and the fact that $\calc_1^\simeq\subset \calc_1$ consists of entire path components; so surjectivity of this map is equivalent to $\pi_0$-surjectivity. This shows (ii) for $\calc$. The claim for $P$ is a direct consequence of (i).
\end{proof}

%Hence the following Theorem is a generalization of the local-to-global principle for topological categories as stated above.

\begin{theorem}[Local-to-global principle, semi-Segal version]\label{thm:main_Segal}
Let 
\[\cald'\xrightarrow{F} \cald \xleftarrow{P} \calc\]
be a diagram of weakly unital semi-Segal spaces and weakly unital maps. If  $P$ a level, cartesian and cocartesian fibration, then the  pull-back diagram
\begin{equation}\label{eq:main_diagram}
\xymatrix{
\cald'\times_\cald \calc \ar[r] \ar[d]^{P'} & \calc \ar[d]^P
\\
\cald' \ar[r]^F & \cald
}
\end{equation}
induces a homotopy pull-back diagram of spaces
\[
\xymatrix{
B(\cald'\times_\cald \calc) \ar[r] \ar[d]^{BP'} & B\calc \ar[d]^{BP}
\\
B\cald' \ar[r]^{BF} & B\cald,
}
\]
through application of the geometric realization functor.
\end{theorem}

Here, $\cald'\times_\cald \calc$ denotes the pull-back in the category of semi-simplicial spaces, given by the pull-back of spaces in each simplicial degree.

Finally, we give a version of the local-to-global principle in the context of simplicial sets. This result is also a consequence of recent work by Ayala--Francis \cite[5.17]{AF}; we will supply a direct proof in the next section. 

Recall that a map $p\colon X\to Y$ between simplicial sets is called an \emph{inner fibration} if for every $n>1$ and every $0<i<n$, every solid diagram
\begin{equation}\label{eq:horn_lifting}
\xymatrix{
 {\Lambda_i^n} \ar[d]_{\mathrm{incl}} \ar[r] & X \ar[d]^p\\
 \Delta^n \ar[r]^\sigma \ar@{.>}[ru]  & Y
}
\end{equation}
has a dotted diagonal lift. A 1-simplex $f$ in $X$ is called \emph{$p$-cartesian} if every solid diagram \eqref{eq:horn_lifting} has a dotted diagonal lift, provided $i=n>1$ and the last edge of $\sigma$ is $f$. $p$ is a \emph{cartesian fibration} if it is an inner fibration and if for every $g\in Y_1$ and any lift $c$ of $d_0(g)$ through $p$, there exists a lift $f\in X_1$ of $g$ through $p$ such that $d_0(f)=c$; dually $p$ is a \emph{cocartesian fibration} of $p\op$ is a cartesian fibration.

\begin{theorem}[Local-to-global principle, simplicial set version]\label{thm:main_quasicat}
A cartesian and cocartesian fibration $P\colon X\to Y$ between simplicial sets is a realization-fibration; that is, for any map $Y'\to Y$ of simplicial sets, the pull-back diagram
\begin{equation*}%\label{eq:main_diagram}
\xymatrix{
Y'\times_Y X \ar[r] \ar[d]^{P'} & X \ar[d]^P
\\
Y' \ar[r]^F & Y
}
\end{equation*}
is a homotopy pull-back diagram of simplicial sets. 
\end{theorem}

\section{Proof of the local-to-global principle}\label{sec:proof_ltgp}

\subsection{The simplicial set case}

The proof in the simplicial set version is an easy combination of two well-known results, the first of which says that simplicial (forward) homotopies may be lifted along cocartesian fibrations.

\begin{lemma}\label{lem:path_lift_in_cocartesian_fibration}
Let $p\colon X\to Y$ be a cocartesian fibration, and let $J\to I$ be an injective map of simplicial sets. If in the solid commutative diagram
\[\xymatrix{
 I\times \{0\}\cup_{J\times \{0\}} J\times \Delta^1 \arinclinv[d] \ar[rr]^(0.6){F_0} && X \ar[d]^p\\
 I\times \Delta^1   \ar@{.>}[rru]^F \ar[rr]^f && Y
}\]
the map $F_0$ sends any edge $(j,0)\to (j,1)$ (for $j\in J_0$) to a $p$-cocartesian edge, then there exists a dotted lift $F$ which keeps the diagram commutative, and which in addition sends any edge $(i,0)\to (i,1)$ (for $i\in I_0$) to a $p$-cocartesian edge.
\end{lemma}

\begin{proof}
By induction and a colimit argument, it is enough to consider the case where $J\to I$ is the inclusion of $\dell\Delta^n$ into $\Delta^n$. If $n=0$, then such a lift exists by definition of  a cocartesian fibration. If $n>0$, then this is \cite[2.4.1.8]{Lurie_HTT}. 
\end{proof}

For $P\colon X\to Y$, and a simplex $\sigma\colon \Delta^n\to Y$ of $Y$, we write $X\vert_\sigma:=\Delta^n\times_Y X$.

\begin{corollary}\label{cor:fiber_transport}
Let $P\colon X\to Y$ be  a cocartesian fibration of simplicial sets. For any simplex $\sigma$ in $Y$, with last vertex $\ell(\sigma)$, the inclusion map
\[X\vert _{\ell(\sigma)}\to X\vert _{\sigma}\]
is a weak homotopy equivalence. 
\end{corollary}

Indeed, a simplicial homotopy inverse is obtained by lifting the standard simplicial homotopy $\Delta^n\times \Delta^1\to \Delta^n$ that deforms the $n$-simplex to its last vertex, where $n=\vert \sigma\vert$. 

%\begin{proof}
% Denote by $h\colon [n]\times [1]\to [n]$ the standard homotopy that collapses the $n$-simplex to its last vertex. As $p$ is a cocartesian fibration, so is its pull-back $p'\colon [n]\times_\cald \calc\to \cald$ and by Lemma \ref{lem:path_lift_in_cocartesian_fibration} there is a lift in the diagram
% \[\xymatrix{
%  ([n]\times_\cald \calc) \times \{0\} \cup_{([0]\times_\cald \calc)\times \{0\}} ([0]\times_\cald \calc)\times [1] 
%     \ar[rr]^(.6){\id\cup (I_\ell\circ\proj)} \arinclinv[d] 
%     && [n]\times_\cald \calc \ar[d]^{p'} \\
%  ([n]\times_\cald \calc)\times [1] \ar[r]_{p'\times \id}\ar@{.>}[rru]^H 
%     &[n]\times [1] \ar[r]_h
%     & [n]
% }\]
% 
% The restriction of $H$ to $([n]\times_\cald \calc)\times \{1\}$ defines a map $r\colon [n]\times_\cald \calc\to [0]\times_\cald \calc$ so that $H$ is a homotopy from $\id_{[n]\times_\cald \calc}$ to $\bar{f}\circ r$ which is stationary over $[0]\times_\cald \calc$. Hence $\bar{f}$ is an equivalence.
%\end{proof}

\begin{corollary}\label{cor:fiber_transport_back_and_forth}
Let $P\colon X\to Y$ be  a  cartesian and cocartesian fibration of simplicial sets. Then, for any simplex $\sigma$ of $Y$, the maps induced by inclusion of the first and last vertex respectively,
\[X\vert_{i(\sigma)} \to X\vert_\sigma \leftarrow X\vert_{\ell(\sigma)}\]
are weak homotopy equivalences. 
\end{corollary}

Thus, the simplicial version of the local-to-global principle follows from:

\begin{lemma}\label{lem:realization_fibration}
Let $P\colon X\to Y$ be  a map of simplicial sets such  that for each simplex $\sigma$ of $Y$, the inclusion maps $Y\vert_{i(\sigma)}\to Y\vert_\sigma \leftarrow Y\vert_{\ell(\sigma)}$ are weak homotopy equivalences. Then, $P$ is a realization-fibration.
\end{lemma}

\begin{proof}
This follows readily from \cite[Lemma 1.4.B]{Waldhausen(1985)}, which Waldhausen deduces from Quillen's Theorem B. We find it more transparent to give a direct proof as follows (see also \cite{Rezk}): It is enough to show that for each $y\in Y_0$, the inclusion of the fiber $X\vert_y$ into the homotopy fiber is an equivalence. A colimit argument reduces to the case where $n=\dim(Y)<\infty$, and we prove this by induction on $n$. 

The case $n=0$ is trivial. For the induction step, we let $Y^{(n-1)}\subset Y$ the $(n-1)$-skeleton, and let $X^{(n-1)} := X\times_Y Y^{(n-1)}$. Then the map  $P$ is the map induced on horizontal push-outs in the following diagram:
\begin{equation}  \label{eq:inductive_step_diagram}\xymatrix{
\coprod_i \Delta^n \times_Y X \ar[d]
  & \coprod_i \dell\Delta^n \times_Y X\ar[l] \ar[r] \ar[d]
  & X^{(n-1)} \ar[d]
 \\
 \coprod_i \Delta^n 
  & \coprod_i \dell\Delta^n \ar[l] \ar[r]
  & Y^{(n-1)}
}\end{equation}
By inductive assumption applied to $Y^{(n-1)}$ and $\coprod_i \dell \Delta^n$, the right square is a homotopy pull-back. To analyze the left square, we consider each summand separately and further restrict along a first or last vertex $\Delta^0\subset \Delta^n$:
\[\xymatrix{
 X\times_Y \Delta^0 \ar[r] \ar[d]
 & X\times_Y \dell\Delta^n \ar[r] \ar[d]
 & X\times_Y \Delta^n \ar[d]
 \\
 \Delta^0 \ar[r] 
 & \dell\Delta^n \ar[r]
 & \Delta^n
}\]
In this diagram, the total square is homotopy cartesian by assumption and the left one by induction hypothesis. Since every connected component of $\dell\Delta^n$ is hit by the first or the last vertex, this shows that the right square is also homotopy cartesian. 

Thus, both  squares in \eqref{eq:inductive_step_diagram} are homotopy cartesian. It follows (see \cite[1.7]{Segal}) that also the pull-back square
\[\xymatrix{
 X^{(n-1)} \ar[r] \ar[d] & X \ar[d]\\
 Y^{(n-1)} \ar[r] & Y
}\]
is homotopy cartesian, and we complete the proof of the inductive step by invoking the inductive assumption.
\end{proof}

We close this section with a remark on the relation of this simplicial version of the local-to-global principle to Quillen's Theorem B. I am grateful to Thomas Nikolaus for pointing this out to me.

By Corollary \ref{cor:fiber_transport} we may define, for each 1-simplex $f\colon c\to c'$ in $Y$, a fiber transport
\[f_*\colon X\vert_c \to X\vert_f \xleftarrow\simeq X\vert_{c'}\]
if $p$ is a cocartesian fibration; dually a fiber transport
\[f^*\colon X\vert_{c'} \to X\vert_f \xleftarrow\simeq X\vert_c\]
if $p$ is a cartesian fibration.

\begin{theorem}\label{thm:cocartesian_Quillen_B}
Let $P\colon X\to Y$ be a map of simplicial sets. Suppose that $P$ a cocartesian fibration, such that for each 1-simplex $f$ in $Y$, the fiber transport $f_*$ is a weak equivalence. Then, $P$ is a realization-fibration. 
\end{theorem}

Of course, the same conclusion then holds if $P$ is a cartesian fibration, such that all fiber transports $f^*$ realize to weak equivalences; for the dual of a realization-fibration is again a realization-fibration.

The proof of this Theorem consists in noting that the inclusions $Y\vert _{f(\sigma)}\to Y\vert_\sigma$ are equivalences (by Corollary \ref{cor:fiber_transport}), but also for each 1-simplex $\sigma$ of $Y$, the inclusion $Y\vert_{\ell(\sigma)} \to Y\vert_\sigma$ is an equivalence (by the  hypothesis on fiber transport). But this still implies the statement of Corollary \ref{cor:fiber_transport_back_and_forth}, so that we may apply Lemma \ref{lem:realization_fibration} just as above. 

Theorem \ref{thm:cocartesian_Quillen_B} implies Quillen's  Theorem B \cite{Quillen} as follows: For an arbitrary functor $F\colon \calc\to \cald$ between ordinary categories, we form a new category $F/\cald$ as the pull-back category
\[\xymatrix{
F/\cald \ar[r] \ar[d] & \cald^{[1]} \ar[d]^{\mathrm{eval}_0} \\
\calc \ar[r]^F &\cald
}\]

Then the composite projection
\[F/\cald\to \cald^{[1]} \xrightarrow{\mathrm{eval}_1} \cald\]
is a cocartesian fibration, with fiber over $d$ the translation category $F/d$; the fiber transport along $f\colon d\to  d'$ is homotopic to the canonical map $F/d\to  F/d'$ induced by postcomposition with $f$. If we assume that these maps all realize to weak equivalences, then  Theorem \ref{thm:cocartesian_Quillen_B} implies that the canonical sequence
\[F/d \to F/\cald \to  \cald\]
is a homotopy fiber sequence. On the other hand the projection
\[F/\cald \to \calc\]
is a  cartesian fibration over $\calc$, with contractible fibers $f(c)/\cald$, so that it is a weak equivalence again by Theorem \ref{thm:cocartesian_Quillen_B}. The projection $F/\cald\to \cald^{[1]}$ defines  natural transformation in the triangle 
\[\xymatrix{
 & F/\cald \ar[ld] \ar[rd]\\
 \calc \ar[rr]^F && \cald
}\]
so that it is homotopy-commutative (after realization). Therefore, for each object $d$ of $\cald$ there is a fibration sequence
\[B(F/d) \to B\calc\xrightarrow{BF} B\cald,\]
which is the conclusion of Quillen's Theorem B.

\subsection{The semi-Segal case}

We deduce the semi-Segal, and hence the topological version of the local-to-global principle, from the simplicial one. The main input is the existence of simplicial structures from \cite{degen}. (Note that the free addition of identity morphisms, while preserving the homotopy type of the classifying space, usually does not preserve (co-)cartesian morphisms.) 

We start by recalling the concept of a Reedy fibration. For a semi-simplicial space $X$, and a semi-simplicial set $A$, we denote
\[X_A:= \lim_{[n] \to A} X_n, \]
where the limit runs over the category of simplices of $A$, in the semi-simplicial sense. 
An equivalent description of $X_A$ is the mapping space from $A$ to $X$, that is, the set of natural transformations $A\to X$, equipped with the subspace topology of  $\prod_n \map(A_n, X_n)$; here $A_n$ is to be considered as a discrete topological space. 
With this notation, the canonical map $X_n\to X_{[n]}$ is an isomorphism, which we will view as an identification.
%, and the Segal map for a semi-simplicial space is the map $X_n\to X_{L_n}$ induced by the inclusion $L_n\to [n]$. 
(Here, and  throughout this section, we denote by $[n]$ the semi-simplicial $n$-simplex, \emph{i.e.}, presheaf represented by $[n]$, and by $\dell[n]\subset [n]$ its boundary.) Also, the construction $X_A$ is covariantly functorial in $X$ and contravariantly functorial in $A$.

\begin{definition}
A map $P\colon X\to Y$ between semi-simplicial spaces is a \emph{Reedy fibration} if for all $n$, the map 
\[X_n = X_{[n]} \to  Y_n\times_{Y_{\dell[n]}} X_{\dell[n]}\]
is  a fibration.  $X$ is \emph{Reedy fibrant} if the map to the terminal object $*$ is a Reedy fibration.
\end{definition}

%Note that each map of semi-simplicial sets is automatically a Reedy fibration, as a finite limit of discrete spaces is discrete again. 
It is easily seen that the class of Reedy fibrations is closed under composition and pull-back. 
Also, the standard techniques show that if $P$ is a Reedy fibration, then for all inclusions of semi-simplicial sets $A\to B$, the induced map $X_B \to Y_B\times_{Y_A} X_A$ is a fibration.

\begin{lemma}\label{lem:Reedy_fibrant_replacement}
Any map $P\colon X\to Y$ can be factored into a level equivalence $X\to X^f$, followed by a Reedy fibration $X^f\to Y$. 
\end{lemma}

\begin{proof}
If $P\colon X\to Y$ satisfies the Reedy condition up to simplicial degree $n-1$, then we define a new semi-simplicial space $X'$ as follows: Choose a factorization 
\[X_n\xrightarrow[\simeq]{i} X'_n \xrightarrow{p} Y_n\times_{Y_{\dell[n]}} X_{\dell[n]}\]
of the Reedy map where $i$ is the a weak equivalence, and $p$ is a fibration. Letting $X'_k:=X_k$ for $k\neq n$ we obtain a new semi-simplicial space $X'$;  $i$ and $p$ induce semi-simplicial maps $I\colon X\to X'$ and $P'\colon X'\to Y$ whose composite is $P$. By construction, $P'$ satisfies the Reedy condition up to degree $n$. 

Now the claim follows by doing this construction iteratively, starting at $n=0$, and taking the colimit. 
\end{proof}

The following result implies that applying Reedy fibrant replacement does not destroy the hypotheses in the local-to-global principle.

\begin{lemma}\label{lem:cartesian_preserved_under_level_equivalences}
Let 
\[\xymatrix{
 \calc\ar[r]^F_\simeq \ar[d]^P & \calc' \ar[d]^{P'}\\
 \cald\ar[r]^G_\simeq & \cald'
}\]
be a commutative diagram of semi-Segal spaces, with horizontal maps level equivalences. Then:
\begin{enumerate}
 \item If one of $P$ and $P'$ is a weakly unital map of weakly unital semi-Segal spaces, then so is the other.
 \item   If one of $P$ and $P'$  is a cartesian (resp.\ cocartesian) fibration, then so is the other.
\end{enumerate}
\end{lemma}

\begin{proof}
A 1-simplex in $\calc$ is an equivalence if and only if its image under $F$ is an equivalence in $\calc'$. Since the subspace of equivalences $\calc_1^\simeq\subset \calc_1$ is a collection of path components, it follows that $F$ induces an equivalence $\calc_1^\simeq \to (\calc')_1^\simeq$. Therefore, under the hypotheses of the Lemma, there is a commutative diagram
\[\xymatrix{
  \pi_0\calc_1^\simeq \ar[d]^{d_0} \ar[r]^F_\cong &  \pi_0 (\calc')_1^\simeq \ar[d]^{d_0} \\
  \pi_0\calc_0 \ar[r]^F_\cong &  \pi_0\calc'_0
}\]
so $\calc$ is weakly unital if and only if $\calc'$ is. Similarly, from the commutative  diagram
\[\xymatrix{
 \pi_0 \calc_1^\simeq \ar[d]^P \ar[r]^F_\cong & \pi_0 (\calc')_1^\simeq \ar[d]^{P'}\\
 \pi_0\cald_1 \ar[r]^G_\cong & \pi_0 (\cald')_1
}\]
we deduce that $P$ is weakly unital if and only if $P'$ is. This proves (i). Part (ii) is proven by a very similar argument which we leave to the reader.
\end{proof}

For later use, we record the following consequence:

\begin{corollary}\label{cor:cocartesian_fibration_stable_under_pullback}
(Co-)cartesian fibrations are stable under level-wise homotopy pull-back.
\end{corollary}

In fact, this follows directly from the  definitions in case the the map under question is also a Reedy fibration and the pull-back is strict; so the general case follows from Reedy fibrant replacement and Lemma \ref{lem:cartesian_preserved_under_level_equivalences}.

Now, the following allows to translate the weakly unital semi-Segal setting into the setting of simplicial sets. We denote by $(-)^\delta$ the functor from semi-simplicial spaces to semi-simplicial sets that forgets the topology. We also note that the standard notions of inner fibrations, (co-)cartesian edges, and (co-)cartesian fibrations between simplicial sets make perfectly sense for maps of semi-simplicial sets, as they are given by horn filling conditions. 

\begin{lemma}\label{lem:translation_Segal_quasicat}
Let $P\colon \calc\to \cald$ be Reedy fibration of Reedy fibrant semi-Segal spaces. Then,
\begin{enumerate}
 \item $P^\delta$ is an inner fibration between inner fibrant semi-simplicial sets.
 \item If $f\in \calc_1$ is $P$-cartesian (in the semi-Segal sense), then it is also $P^\delta$-cartesian (in the sense of semi-simplicial sets).
 \item If $P$ is a (co-)cartesian fibration then so  is $P^\delta$.
\end{enumerate}
If, furthermore, $\calc$, $\cald$, and $P$ are weakly unital, then 
\begin{enumerate}\setcounter{enumi}{3}
 \item $\calc^\delta$ and $\cald^\delta$ admit simplicial structures such that $P^\delta$ is simplicial. In more detail, given any simplicial structure on $\cald^\delta$, there exists a simplicial structure on $\calc^\delta$ such that $P^\delta$ is simplicial. 
 %We also may arrange that the simplicial structure on $\calc^\delta$ extends a given simplicial structure on a given semi-simplicial subset of $\calc^\delta$.
 \item The canonical maps $\calc^\delta\to \calc$ and $\cald^\delta\to \cald$ induce a weak equivalence on classifying spaces.
\end{enumerate}
\end{lemma}

\begin{proof}
It follows from  \cite[3.4]{degen} that the map $\calc_n \to \cald_n\times_{\cald_{\Lambda^n_i} } \calc_{\Lambda^n_i}$ is an acyclic fibration of spaces for each $0<i<n$; therefore it is surjective. Hence $P$, and equally $P^\delta$, have the right lifting property against all inner horn inclusions, so $P^\delta$ is an inner fibration. Applying this to the projection $\cald\to *$ shows that $\cald^\delta$ is inner fibrant. This shows (i). Part (ii) follows by the  same reasoning from \cite[3.6]{degen}. To see  (iii), we note that 
\[(P,d_0)\colon \calc_1^{\cart P}\to \calc_0\times_{\cald_0} \cald_1\]
is a $\pi_0$-surjective fibration and therefore surjective, and remains surjective after applying $(-)^\delta$, so that the claim follows from (i) and (ii).

Let us prove (iv). We first show that every object $c$ of $\cald$ admits a homotopy idempotent self-equivalence, in other words, a two-simplex $\tau$, all whose vertices are $c$, and all whose faces agree and are equivalences. Indeed, since
\[d_0\colon \cald_1^\simeq \to \cald_0\]
is a $\pi_0$-surjective fibration, it is surjective, and there exists an equivalence $f\colon d\to c$ with target $c$. Then $f$ is cocartesian for the projection $\cald\to *$, and by part (ii) also cocartesian for $\cald^\delta\to *$, in the quasicategorical sense. We use this cocartesian property of $f$ to find a 2-simplex $\sigma\in \cald_2$,
\[\xymatrix{
 d \ar[rr]^f \ar[rd]_f && c \ar@{.>}[ld]^u\\
  & c
}\]
such that $d_1\sigma=d_2\sigma=f$. Now equivalences in semi-Segal  spaces are easily seen to satisfy the 2-out-of-3 property, so $u:=d_0 f$ is still an equivalence (in $\cald$, therefore in $\cald^\delta$ again by part (ii)). Applying the cocartesian property once more, we obtain a 3-simplex
\[\xymatrix{
 d \ar[rd]^(.7)f \ar[rrrd]^f \ar@/_5ex/[rrdd]_f\\
 & c \ar[rr]^u \ar[rd]^u && c \ar[ld]^u\\
 && c
}
\]
whose first three faces are $\sigma$; its last boundary is then a simplex $\tau$ as required. 
%The dual argument works if $c$ is the domain of an equivalence. 

This shows that the semi-simplicial set $\cald^\delta$ fulfills the hypotheses of \cite[Theorem 1.2]{degen} and we conclude that $\cald^\delta$ admits a simplicial structure.%, extending a given one on a given semi-simplicial subset. 

Now, since $P$ is weakly unital, any equivalence of $\calc$ is both $P$-cartesian and $P$-cocartesian. Applying the same argument as above  in the relative situation, we then see that any object $c$ of $\calc$ admits a 2-simplex, all whose vertices are $c$, and all whose faces agree and are $P$-cartesian and $P$-cocartesian, and which maps under $P$ to the degenerate simplex on $c$, in the newly constructed simplicial structure. Then it follows from \cite[Theorem 2.1]{degen} that $\calc^\delta$ admits a simplicial structure such that $P^\delta$ is simplicial. 

It remains to prove (v). By definition, $\cald$ being Reedy fibrant means that each diagram $(0\leq i\leq k, n\geq 0)$ of semi-simplicial spaces
\[\xymatrix{
\Lambda^k_i \times [n] \cup_{\Lambda^k_k\times \dell[n]} \Delta^k\times \dell [n] \ar[rr] \ar[d] && \cald\\
\Delta^k \times [n] \ar@{.>}[rru]
}\]
has  a dotted extension. This again means  that the restriction map (of  semi-simplicial  spaces)
\[\map(\Delta^k, \cald) \to \map(\Lambda^k_i, \cald)\]
has the right lifting property against all inclusions $\dell[n]\to [n]$. Since each inclusion $\Delta^k\to \Delta^{k+1}$ may be obtained by filling in horns, it follows that evaluation at each vertex
\begin{equation}\label{eq:singular_functor_for_semi_Segal_space}
\map(\Delta^k, \cald)\to \map(\Delta^0, \cald) 
\end{equation}
also has the right lifting property against the same set of inclusions. Now it follows from Lemma \ref{lem:cartesian_preserved_under_level_equivalences} that \eqref{eq:singular_functor_for_semi_Segal_space} is a weakly unital map of weakly unital semi-Segal spaces; by part (iv) it follows that the induced map
\[\map(\Delta^k, \cald)^\delta\to \map(\Delta^0, \cald)^\delta\]
may be given a simplicial structure. This simplicial map then has the right lifting property against all simplicial inclusions $\dell\Delta^n_\bullet\to \Delta^n$; therefore realizes to a weak equivalence. 

In other words, the degree-wise singular construction $S_\ast \cald$ is homotopically constant in the $*$-direction, after realizing the semi-simplicial direction of $\cald$. It follows that 
\[\vert \cald^\delta\vert = \vert S_0 \cald\vert \to \vert S_\ast \cald\vert \simeq \vert \cald\vert\]
is an equivalence. This proves the claim for $\cald$; the same reasoning applies to $\calc$. 
\end{proof}

\begin{lemma}\label{lem:Segal_pullback}
In the situation of Theorem \ref{thm:main_Segal}, the (degree-wise) pull-back $\cald'\times_\cald \calc$ is a weakly unital semi-Segal space. 
\end{lemma}

\begin{proof}
The target of the Segal map of the pull-back is
\[(\cald'_1 \times_{\cald_1} \calc_1) \times^h_{\cald'_0\times_{\cald_0} \calc_0} \dots \times^h_{\cald'_0\times_{\cald_0} \calc_0} (\cald'_1 \times_{\cald_1} \calc_1).\]
Switching homotopy limits, this target becomes equivalent to
\[(\cald'_1\times^h_{\cald'_0}\dots \times^h_{\cald'_0} \cald'_1) \times^h_{\cald_1\times^h_{\cald_0}\dots \times^h_{\cald_0} \cald_1} (\calc_1\times^h_{\calc_0}\dots \times^h_{\calc_0} \calc_1)\]
and the Segal condition for the pull-back follows from the Segal conditions for the three terms involved. 

To verify the weak unitality, we first apply Lemma \ref{lem:Reedy_fibrant_replacement} multiple times as to replace the diagram $\cald'\to \cald\leftarrow \calc$ by a level equivalent diagram where all terms are Reedy fibrant and all maps Reedy fibrations. These are still weakly unital maps of weakly unital semi-Segal spaces by Lemma \ref{lem:cartesian_preserved_under_level_equivalences}. Furthermore the pull-back of the new diagram is level equivalent to the pull-back of the old diagram; so that it will be enough to show that the pull-back of the new diagram is weakly unital, again in view of Lemma \ref{lem:cartesian_preserved_under_level_equivalences}. 

It results from this discussion that we may assume that all terms are Reedy fibrant and all maps Reedy fibrations. Let $(d', c)$ be a zero-simplex of the pull-back. As shown in the proof of Lemma \ref{lem:translation_Segal_quasicat}, (iv), there exist weak units (\emph{i.e.,} homotopy idempotent self-equivalences) $\varphi$ of $d'$ and $\psi$ of $c$. Then the images $F(\varphi)$ and $P(\psi)$ are weak units of $F(d')=P(c)$. But any two weak units of the same object are homotopic through some homotopy $H$. Changing $\psi$ by a homotopy obtained by lifting $H$ through the Reedy fibration $P$, we may assume that $F(\varphi)=P(\psi)$ so that $(\varphi, \psi)$ defines a 1-simplex in $\cald'\times_\cald \calc$, which is easily seen to be an equivalence. Therefore any zero-simplex of $\cald'\times_\cald \calc$ is the target of an equivalence.
\end{proof}

\begin{proof}[Proof of Theorem \ref{thm:main_Segal}]
By Lemma \ref{lem:cartesian_preserved_under_level_equivalences}, the hypotheses of the Theorem are invariant under level-equivalences. Therefore, we may assume that all semi-simplicial spaces are Reedy fibrant, and all maps are Reedy fibrations (else  applying Lemma \ref{lem:Reedy_fibrant_replacement} multiple times). 

Now, applying $(-)^\delta$ level-wise at each entry, by Lemma \ref{lem:translation_Segal_quasicat} we obtain diagram of semi-simplicial sets and inner fibrations,
\[\cald^\delta \xrightarrow{F^\delta} \calc^\delta \xleftarrow{P^\delta} \calc^\delta\]
where the right map is a cartesian and cocartesian fibration, and where all terms can be given compatible simplicial structures so that all maps become simplicial. By Theorem \ref{thm:main_quasicat}, the  geometric realization of the pull-back diagram
\[\xymatrix{
(\cald')^\delta \times_{\cald^\delta} \calc^\delta \ar[r] \ar[d] & \calc^\delta\ar[d]^{P^\delta}\\
\cald^\delta \ar[r]^{F^\delta} & \calc^\delta
}\]
is a homotopy pull-back. (Strictly speaking, this is the geometric realization of simplicial sets, but this one is well-known to be equivalence to  the geometric realization of underlying semi-simplicial sets.) The  identity function induces a natural transformation, in the category of semi-simplicial spaces, from the above diagram to the diagram \eqref{eq:main_diagram}. It is an equivalence after realization in each  term, by Lemma \ref{lem:translation_Segal_quasicat}, (v), where we note that the term $\cald'\times_\cald \calc$ fulfills the assumptions of this Lemma by Lemma \ref{lem:Segal_pullback}. Hence, \eqref{eq:main_diagram} is also a homotopy pull-back.
\end{proof}

\section{Cobordism categories of \texorpdfstring{$\cor{k}$}{<k>}-manifolds} \label{sec:genauer}

In this section we use the concept of neat submanifolds, see Appendix \ref{sec:isotopy_extension}. Loosely, we define the $d$-dimensional cobordism category of $\cor{k}$-manifolds as follows: An object is a compact  $M\subset \rc{\infty+d-1}{k}$ which is a smooth, neat submanifold of dimension $(d-1)$. A morphism is a pair $(W,a)$ of $a>0$ and a compact $W\subset [0,a]\times \rc{\infty+d-1}{k}$ which is a smooth neat submanifold of dimension $d$. Such a $(W,a)$ is viewed as a morphism from $\dell_{-}W:= W\cap \{0\}\times \rc{\infty+d-1}{k}$ to $\dell_{+} W:=  W\cap \{a\}\times \rc{\infty+d-1}{k}$.  Composition is given by stacking and adding the real parameters.

Instead of giving this (non-unital) category a topology, we view this as the value at $[0]$ of a simplicial category $\cob_{d, \cor{k},\bullet}$. In simplicial level $n$, an object is a fiber bundle $E\to \Delta^n$ of smooth $\cor k$-manifolds, which is fiberwise $\delta$-neatly embedded into $\Delta^n \times \rc{\infty+d-1}k$, for some $\delta>0$. A morphism is a continuous map $\underline a\colon \Delta^n\to (0,\infty)$ and a fiber bundle over $\Delta^n$ of compact smooth $\cor{k+2}$-manifolds, fiberwise $\delta$-neatly embedded into $\Delta^n\times [0,\underline a]\times \rc{\infty+d-1}{k}$ for some $\delta>0$. Here we use the notation 
\[[0,\underline a]\times X:= \{(t,x)\in \IR\times X\;\vert\; 0\leq t\leq \underline a(\pi(x)) \} \]
for a space $X$ with a map $\pi\colon X\to \Delta^n$.

\emph{Details.} By ``fiber bundle $E\to B$ of smooth $\cor{k}$-manifolds'' we mean a fiber bundle with a reduction of the structure group to $\Aut(M)$, the group of neat and allowable automorphisms of the typical fiber, with $C^\infty$-topology. A fiberwise embedding from one bundle into another is fiberwise smooth and neat if, in local charts, it is given by a continuous map to $\Emb(M,N)$, the  space of neat and allowable embeddings, with $C^\infty$-topology. (Compare Appendix \ref{sec:isotopy_extension}). Of course, the reduction of the structure group for the bundle $\Delta^n\times[0,\underline a]\times  \rc{\infty+d-1}{k}$ is given by the unique chart $\Delta^n\times [0,1]\times \rc{\infty+d-1}{k}$ which rescales fiberwise by $\underline a$.

The simplicial category $\cob_{d, \bullet}:=\cob_{d,\cor{0},\bullet}$ is a simplicial version of the usual $d$-dimensional cobordism category, and $\cob_{d,\dell,\bullet}:=\cob_{d, \cor{1},\bullet}$ is a simplicial version of the $d$-dimensional cobordism category of manifolds with boundaries. We denote by $\cob_{d,\cor{k}}$ the topological category obtained from $\cob_{d,\cor{k},\bullet}$ by geometric realization of objects and morphisms and by $B\cob_{d,\cor{k}}$  its classifying space.

There is an obvious commutative square of functors
\[\xymatrix{
 {\cob_{d,\cor{k}}} \ar[r]^{i_{1}} \ar[d] & \cob_{d,\cor{k+1}} \ar[d]^{\dell_{1}} \\
 \cob^\emptyset \ar[r] & \cob_{d-1, \cor{k}}
}\]
where the upper horizontal functor is the inclusion 
\[i_{1}\colon (M;M_1,\dots, M_k) \mapsto (M; \emptyset, M_1,\dots, M_k)\]
(using a standard inclusion $\IR\to [0,\infty)$ in the relevant coordinate) on objects and morphisms, and the functor $\dell_{1}$ is given by taking the first boundary on objects and morphisms. The lower left entry denotes the subcategory of $\cob_{d-1, \cor{k}}$ on the empty set and empty morphisms. Actually this is isomorphic to the topological semi-group $\vert S_\bullet (0,\infty)\vert $ (the geometric realization on the singular construction on the space $(0,\infty)$, semi-group structure by addition). The lower horizontal functor is the inclusion and the left vertical functor sends $(W,\underline a)$ to $\underline a$.

\begin{theorem}\label{thm:genauer}
The above square becomes a homotopy pull-back square after geometric realization. Since $B\cob^\emptyset$ is contractible, we obtain a fibration sequence
\[B\cob_{d,\cor{k}} \xrightarrow{Bi_{1}} B\cob_{d,\cor{k+1}} \xrightarrow{B\dell_{1}} B\cob_{d-1,\cor{k}}.\]
\end{theorem}

In the case $k=0$ this yields the sequence from Theorem \ref{thm:main_genauer}. --- To prove this theorem, we will verify the criterion of the local-to-global principle for the functor $\dell_{1}$. First both $\cob_{d,\cor k}$ and $\cob_{d-1,\cor{k-1}}$ are weakly unital and $\dell$ is a weakly unital functor, for $M\times [0,1]$ is a weak unit of the object $M$ in either category. The subcategory $\cob_{d-1}^\emptyset$ of $\cob_{d-1}$ consists entirely of weak units in $\cob_{d-1}$ and is clearly locally fibrant. 

Next we show that $\cob_{d-1, \cor{k}}$ is locally fibrant, that is, that the combined source-target map of $\cob_{d-1, \cor{k}}$ is a Serre fibration. Let $W$ be a morphism of $\cob_{d-1,\cor{k}}$. As a consequence of Theorem \ref{thm:isotopy_extension}, the restriction map
\[\Emb(W, \rc{\infty+d-1}{k}\times [0,1]) \to \Emb(M, \rc{\infty+d-1}{k}) \times \Emb(N, \rc{\infty+d-1}{k})\]
is a Kan fibration. (Indeed, identifying $\rc{\infty+d-1}{k}\times [0,1]$ with $\rc{\infty+d}{k+2}-\dell_{k+2}\dell_{k+1} \rc{\infty+d}{k+2}$, we can view $W$ as a neat submanifold of $\rc{\infty+d}{k+2}$. As such, we have a canonical homeomorphism
\[\Emb(W, \rc{\infty+d-1}{k}\times [0,1])\cong \Emb(W, \rc{\infty+d}{k+2})\]
under which the restriction to the boundary of $[0,1]$ above corresponds to restriction to $\Emb(W(A), \rc{\infty+d}{k+2}(A))$, with $A\subset \Delta^{\underline{k+2}}$ the subcomplex generated by the last two faces.)

With $S_\bullet$ the singular construction, we have a commutative square
\begin{equation}\label{eq:embeddings_and_morphisms}
\resizebox{\textwidth}{!}{
\xymatrix{
\coprod_{[W]} \vert S_\bullet \bigl(\Emb(W, \rc{\infty+d-1}{k}\times [0,1])\times (0,\infty)\bigr) \vert  \ar[r] \ar[d] & \coprod_{[M, N]}\vert S_\bullet \Emb(M, \rc{\infty+d-1}{k})\vert \times \vert S_\bullet \Emb(N, \rc{\infty+d-1}{k})\vert \ar[d]\\
\mor\cob_{d-1,\cor{k}} \ar[r] & \ob\cob_{d-1, \cor{k}} \times \ob\cob_{d-1,\cor k}
}
}
\end{equation}
where the upper horizontal map is still a Kan fibration because both $S_\bullet$ and $\vert - \vert$ preserve Kan fibrations. The vertical maps are Kan fibrations by letting $G_\bullet = S_\bullet \Diff(W)$ in the following criterion:

\begin{lemma}\label{lem:simplicial_group_fibration}
Let $G_\bullet$ be  a simplicial group acting simplicially on a simplicial set $X_\bullet$, with level-wise quotient simplicial set $(X/G)_\bullet$. Suppose that the action is free in each simplicial degree. Then the projection map $X_\bullet\to (X/G)_\bullet$ is a Kan fibration.
\end{lemma}

\begin{proof}
Let $x\colon \Delta^n_\bullet \to (X/G)_\bullet$ be a simplicial map and $y'\colon (\Lambda^n_i)_\bullet \to X_\bullet$ be a partial lift. As $X_n\to (X/G)_n$ is surjective, there is a lift $y\colon \Delta^n_\bullet \to X_\bullet$ of $x$. As $G_\bullet$ acts freely in each degree, there is a unique map $g'\colon (\Lambda^n_i)_\bullet \to G_\bullet$ such that 
\[y'= g'\cdot y\vert_{(\Lambda^n_i)_\bullet}.\]
Now any simplicial group is Kan so $g'$ extends to a map $g\colon \Delta^n_\bullet\to G_\bullet$. Then $g\cdot y$ is a lift of $x$ restricting to $y'$.
\end{proof}

Moreover the  vertical maps in \eqref{eq:embeddings_and_morphisms} are surjective by definition. It is a formal consequence that the lower horizontal map is a Kan fibration as well. This proves that $\cob_{d-1,\cor{k}}$ is locally fibrant. A very similar argument shows that the functor $\dell_{1}$ is a local fibration, and hence a level fibration. Thus, to complete the proof of Theorem \ref{thm:genauer}, we are left to show:

\begin{lemma}\label{lem:boundary_map_cocartesian_fibration}
The functor $\dell_1$ is a cartesian and cocartesian fibration. 
\end{lemma}

\begin{proof}
We first show that it is a cocartesian fibration. Let $(W,a)\colon M\to N$ be a cobordism between $(d-1)$-dimensional $\cor k$-manifolds, and $X$ a $d$-dimensional $\cor{k+1}$-manifold so that $\dell_1 X=M$; we will construct a $\dell_1$-cocartesian lift of the homotopic morphism $W':=(N\times [0,a])\circ W$ starting at $X$. (By local fibrancy, it follows easily the general existence of $\dell_1$-cocartesian lifts.) 

Basically the cocartesian lift is obtained by manipulating the corner structure of the manifold $[0,a]\times (X\cup_M W)$: We introduce additional corners at $\{0\}\times M$ and straighten the corners at $\{0\}\times N$; the resulting manifold with corners can then be viewed as morphism from $X$ to $X\cup_M W$ with vertical boundary $W'$. This will be the desired cocartesian lift. 

In more precise terms, we will change the canonical embedding of $[0,a]\times (X\cup_M W)$ into 
\[[0,a]\times \rc{\infty+d}{k+1}\subset \IR_+\times \rc{\infty+d}{k+1}=\rc{\infty+d}{k+2}\] 
by an automorphism of $\rc{\infty+d}{k+2}$. This automorphism will have the effect of changing, at the same time, the corner structure and the decomposition of the boundary of $\rc{\infty+d}{k+2}$; and therefore of any neat submanifold of $\rc{\infty+d}{k+2}$. This technique has been employed, in a very similar situation, in \cite[Appendix]{RS_Hcob}. 

To construct this automorphism, we let 
\[\bfB\colon \IR_+^2\to \IR_+^2\]
be a homeomorphism, which is a diffeomorphism except at $p_0:=(0,0)$ and $p_1:=(0,a)$, such that (see figure \ref{fig:bending_morphism}):
\begin{figure}[t]
 \centering
 \includegraphics[%trim=0cm 20cm 3cm 0cm, clip=true, 
 scale=0.3]{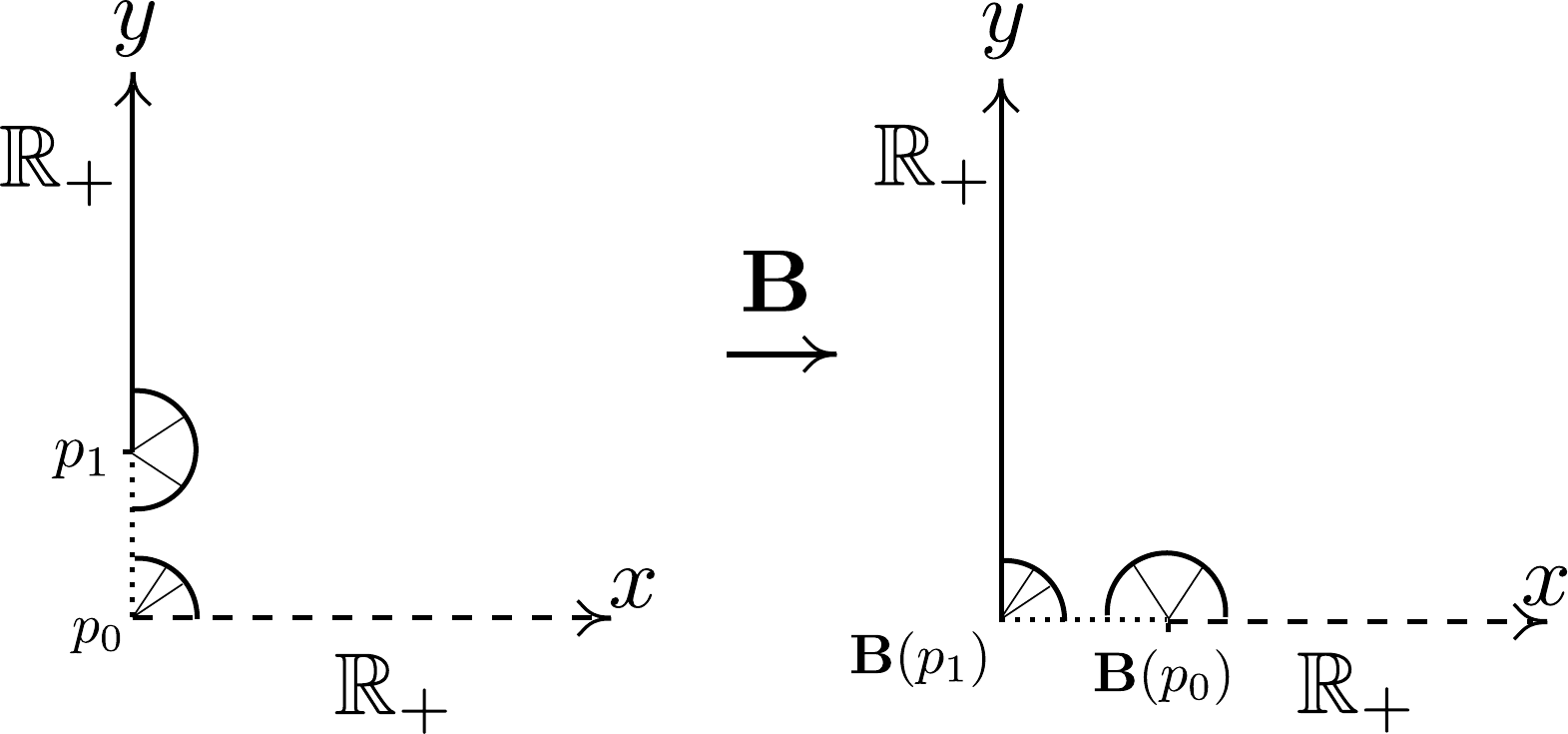}
 \caption{}
 % bending_morphism.pdf: 611x838 pixel, 72dpi, 21.55x29.56 cm, bb=0 0 611 838
 \label{fig:bending_morphism}
\end{figure}

\begin{enumerate}
 \item $\bfB(p_0)=(0,a)$; $\bfB(p_1)=(0,0)$; and $\bfB$ maps the solid/dotted/dashed pieces of the boundary, as displayed on the left in Figure \ref{fig:bending_morphism}, isometrically to the corresponding pieces on the right. 
 \item $\bfB$ is cylindrical in a neighborhood of $\dell \IR_+^2-\{p_0, p_1\}$
 \item In a neighborhood of $p_1$, $\bfB$ preserves the radial coordinate and scales the angle by a diffeomorphism
 \[\lambda\colon [0,\pi]\to [0,\pi/2 ]\]
 which is $\id$ on $[0, \pi/3]$ and $\id-\pi/2$ on $[2\pi/3, \pi]$. Similarly, in a neighborhood of $p_0$, $\bfB$ scales the angle by the inverse of $\lambda$.
 \item The restriction of $\bfB$ to $[a/2,\infty)\times \IR_+$ is translation by $a$ to the right.
\end{enumerate}

\begin{remark*}
A simpler condition in (iii) would be that $\lambda$ is just division by two, so that $\bfB$ just scales the angle by a factor of 2 in a neighborhood of $(0,a)$. However, this condition is inconsistent with requirement (ii). Requirements (ii) and (iii) together guarantee that the image of a neat submanifold under $\bfB$ is again a neat submanifold. 
\end{remark*}

To construct a map $\bfB$ with the required properties, we use an auxiliary homeomorphism 
\[\Psi\colon \IR_+\times \IR_+\to \IR\times \IR_+\]
which is a diffeomorphism except at the corner point $(0,0)$, such that
\begin{enumerate}
 \item in a neighborhood of $(0,\infty)\times \{0\}$, $\Psi$ is the identity;
 \item in a neighborhood of $\{0\}\times (0, \infty)$, $\Psi$ is rotation by $\pi/2$;
 \item in a neighborhood of the corner point, $\Psi$ preserves the radial coordinate and scales the angle by the inverse of $\lambda$;
 \item $\Psi$ is the identity on $[a/2, \infty)\times \IR_+$.
\end{enumerate}
Such a map $\Psi$ was essentially constructed in \cite[Appendix]{RS_Hcob}; in fact it can be obtained from the map $\Phi\colon [0,1]\times \IR_+\to \IR_+\times \IR_+$ in loc.\ cit.\ by first extending, via the identity, to a map
\[(-\infty, 1]\times \IR_+\to \IR\times \IR_+\]
and then flipping and rescaling the first coordinate. Then, a map $\bfB$ as required is obtained as the composite map that first applies $\Psi$, then shifts by $a$ in the first coordinate, and then applies $\Psi\inv$.

We call $\bfB$ the bending homeomorphism. It induces a bending homeomorphism (denoted by the same letter)
\[\bfB\colon \rc{\infty+d}{k+2}\to \rc{\infty+d}{k+2}\]
by applying $\bfB$ in the first two coordinates of the factor $\IR^k_+$ and the identity in the other ones. 

Applying $\bfB\inv$ to the subspace  
\[X\cup W\subset [0,a]\times \rc{\infty+d}{k+1}\subset \IR_+\times \rc{\infty+d}{k+1}=\rc{\infty+d}{k+2}\]
yields a subspace $Y\subset \rc{\infty+d}{k+2}$, which is a neatly embedded submanifold of $0\times \rc{\infty+d-1}{k+1}$.

Then $C:= \bfB([0,a]\times Y)$ is a neat submanifold of $[0,2a]\times \rc{\infty+d-1}{k+1}$. Hence $C$ defines a morphism in $\cob_{d, \cor{k+1}}$ and we claim that is indeed cocartesian. In other words, we claim that for any object $Z$ of $\cob_{d,\cor{k+1}}$, with $P:=\dell_{1} Z$, the diagram
\[\xymatrix{
\cob_{d, \cor{k+1}}(Y, Z) \ar[rr]^{-\circ C} \ar[d]^{\dell_{1}}
  && \cob_{d,\cor{k+1}}(X,Z) \ar[d]^{\dell_{1}} 
\\
\cob_{d-1, \cor k}(N, P) \ar[rr]^{-\circ W'} 
  && \cob_{d-1,\cor{k}}(M,P)
}\]
is a homotopy pull-back. We prove the equivalent assertion that the induced map on all vertical fibers
\begin{equation}\label{eq:cocartesian_property}
\cob_{d,\cor{k+1}}(Y, Z)/V \xrightarrow{-\circ C} \cob_{d, \cor{k+1}}(X,Z)/(V\circ W') 
\end{equation}
is an equivalence.

By the special form of $Y=X\cup_M W$, applying $\bfB$ induces a homeomorphism 
\[\bfB\colon \cob_{d, \cor{k+1}}(Y, Z)/V \xrightarrow\simeq \cob_{d, \cor{k+1}}(X,Z)/(V\circ W).\]
If $V$ has length at least $a/2$, then the triangle
\[\xymatrix{
 \cob_{d,\cor{k+1}}(Y, Z)/V \ar[rr]^{-\circ \bfB([0,a]\times Y)} \ar[rrd]_{-\circ ([0,a]\times Y)}
   && \cob_{d,\cor{k+1}}(X, Z)/(V\circ W')
 \\
 && \cob_{d,\cor{k+1}}(Y, Z)/(V\circ (N\times [0,a]) \ar[u]_\bfB^\cong
}\]
is commutative by property (iv) of the map $\bfB$. As $[0,a]\times Y$ is a weak unit, the diagonal map is an equivalence, from which we conclude that the horizontal map is an equivalence as well. 

This shows that the map \eqref{eq:cocartesian_property} is an equivalence provided $V$ has length at least $a/2$. But any morphism is homotopic to one of length $a/2$, so that \eqref{eq:cocartesian_property} is an equivalence for all $V$.

It follows that our functor $\dell_1$ is a cocartesian fibration, and also a cartesian fibration, because $\cob_{d,\cor{k}}\cong \cob_{d,\cor{k}}\op$ by reflecting everything in the last coordinate.
\end{proof}

\subsection*{Tangential structures}

Our proof of Theorem \ref{thm:genauer} also has a generalization to cobordism categories with tangential structures. We give the relevant definitions and results and show how to modify the proof. Recall from \cite{Genauer} that a \emph{$\cor k$-space} is a $\power{\underline k}$-shaped diagram in the category of topological spaces and continuous maps, where $\power{\underline k}$ denotes the power set of $\underline k:=\{1,\dots, k\}$. A \emph{$\cor k$-vector bundle} is $\power{\underline k}$-shaped diagram in the category of vector bundles and fiberwise linear maps. A  \emph{fiberwise linear map} $\xi\to \eta$ between $\cor k$-vector bundles $\xi$, $\eta$ is a natural transformation between the functors $\xi$ and $\eta$. Such a fiberwise linear map is called a \emph{$\cor k$-bundle map} if each $\xi(A)\to \eta(A)$ is a bundle map of ordinary vector bundles, that is, a linear isomorphism in each fiber. 

\begin{definition}
A \emph{collar} on a $\cor k$-vector bundle $\xi$ is a collection of bundle maps (isomorphisms in each fiber)
\[c_{AB}\colon \eps^{B-A}\oplus \xi(A) \to \xi(B), \quad A\subset B\subset \underline k\]
such that
\begin{enumerate}
 \item each $c_{AB}$ extends the structure map $\xi(A)\to \xi(B)$, and
 \item for any $A\subset B\subset C\subset \underline k$, the following triangle commutes:
 \[ \xymatrix{
 \eps^{C-A}\oplus \xi(A) 
    \ar[rr]^{\id_{\eps^{C-B}}\oplus c_{AB}} 
    \ar[rrd]_{c_{AC}}
  && \eps^{C-B}\oplus \xi(B) 
    \ar[d]^{c_{BC}} 
  \\
   &&  \xi(C).
 }\]
\end{enumerate}
A $\cor k$-bundle map $f\colon \xi\to \eta$ between collared $\cor k$-bundles is called \emph{collared} if for each $A\subset B\subset\underline k$, the resulting square  commutes:
\[\xymatrix{
 \eps^{B-A}\oplus\xi(A) \ar[r] \ar[d]^{\id\oplus  f(A)}  & \xi(B) \ar[d]^{f(B)}
 \\
 \eps^{B-A}\oplus\eta(A) \ar[r] & \eta(B).
}\]
The category of collared $\cor k$-vector bundles and collared maps will be denoted by $\colbun{k}$. The \emph{rank} of a collared $\cor k$-vector bundle $\xi$ is defined to be the  rank of the ordinary vector bundle $\xi(\underline k)$.
\end{definition}

% \begin{remark*}
% It is to be expected that this agrees with the notion of ``geometric'' vector bundle in \cite[Definition 3.2]{Genauer}, in the sense that requirements (i) and (ii) are the precise meaning of the word ``canonical'' from that definition.
% \end{remark*}

\begin{example}\label{ex:collared_bundles}
 \begin{enumerate}
  \item The tangent bundle of a neat $\cor k$-manifold $M\subset \rc n k$.
  \item For a vector bundle $\xi$ over $X$, there is a canonical extension to a collared $\cor k$-vector bundle, still denoted by $\xi$, such that $\xi(\underline k)\cong \xi$. The bundle $\xi(A)$ is, up to unique isomorphism, characterized by requiring a bijection between the sets of bundle maps
  \[\map(\eta, \xi(A)) \cong \map (\eta\oplus \eps^{\underline k-A}, \xi)\]
  which is natural in the vector bundle $\eta$. It follows that for any $\cor k$-vector bundle $\eta$, evaluation at $\underline k$ induces a bijection
  \[\map(\eta, \xi) \to \map(\eta(\underline k), \xi(\underline k))\]
  between $\cor k$-bundle maps $\eta\to \xi$ and bundle maps $\xi(\underline k)\to \eta(\underline k)$. 
  %Explicitly, the base space of $\xi(A)$ is the fiber bundle over the base space of $\xi$, whose fiber over $b$ consists of pairs $(L, f)$ where $L\subset \xi_b$ is a linear subspace, and f\colon L\oplus \IR^{\underline k- A}\to \xi_b$ is a linear isomorphism extending the inclusion. The total space of $\xi(A)$ consists of quadruples $(b, L, f, x)$ where in addition, $x$ is an element in $L$.
  \end{enumerate}
\end{example}

Now let $\theta$ be a collared $\cor k$-vector bundle of rank $d$. We define $\cob_{\theta}$, the $\theta$-cobordism category of $\cor k$-manifolds as follows, for simplicity only in simplicial degree 0: A morphism is a morphism $W$ of $\cob_{d,\cor k}$ together with a collared bundle map $l_W\colon TW\to \theta$. An object is an object $M$ of $\cob_{d, \cor k}$ together with a collared bundle map $l_M\colon \eps\oplus TM\to \theta$. The source/target of $(W, l_W)$ is $(\dell_\pm W, l_{\dell_\pm W})$, where we identify the tangent bundle of an interval with the trivial bundle $\eps$ by sending $\dell/\dell x$ to $1$.

For a (collared) $\cor {k+1}$-vector bundle $\theta$, we define:
\begin{enumerate}
 \item $\dell_{1}\theta$ as the first face, viewed as a (collared) $\cor k$-bundle -- that is, restriction along the embedding 
 \[j_1\colon \power{\underline k}\subset \power{\underline{k+1}}, \quad A\mapsto A+1.\]
 \item $i_{1}^*\theta$ to be the (collared) $\cor k$-vector bundle where the first face is omitted -- that is, restriction of the diagram along the embedding
 \[i_1\colon \power{\underline k}\to \power{\underline{k+1}}, \quad A\mapsto A+1 \cup\{1\}.\]
\end{enumerate}
Note that the rank of $\theta$ does not change under $i_{1}^*$, but drops under $\dell_{1}$ by one. 

\begin{theorem}\label{thm:genauer_theta}
For any collared $\cor k$-bundle, the square
\[\xymatrix{
 {\cob_{i^*_{1}\theta}} \ar[r]^{i_{1}} \ar[d] & \cob_\theta \ar[d]^{\dell_{1}} \\
 \cob^\emptyset \ar[r] & \cob_{\dell_{1}\theta}
}\]
becomes a homotopy pull-back square after geometric realization. Since $B\cob^\emptyset$ is contractible, we obtain a fibration sequence
\[B\cob_{i^*_{1}\theta} \xrightarrow{Bi_{1}} B\cob_\theta \xrightarrow{B\dell_{1}} B\cob_{\dell_{1}\theta}.\]
\end{theorem}

The proof is a modification of the proof of Theorem \ref{thm:genauer}. Let us show how put a $\theta$-structure  on the $\dell_1$-cocartesian morphism $C$ constructed above from the datum of $W$ and $X$, provided that $W$ comes equipped with a $\dell_1\theta$-structure, and $X$ with a $\theta$-structure, compatible over $M$. 

% To this end, we recall that $C$ is obtained from $W$ and $X$ by first taking the union (of subsets) of $W$ and $M$ inside $\rc{\infty+d}{k+2}$, then apply $\bfB\inv$ to obtain a $\cor{k+1}$-manifold $Y$, then take a product with $[0,a]$, and apply $\bfB$ as to obtain $C$. 
Roughly speaking, one observes that the the union $W\cup X$ inside $\rc{\infty+d}{k+2}$ canonically inherits a  tangential structure, to which we apply the differential of $\bfB\inv$ (actually, this differential does not exist at the non-smooth points of $\bfB$, so we consider  a small modification near the non-smooth points, called $\tilde D \bfB\inv$). This yields a  tangential structure on $Y=\bfB\inv(W\cup X)$  and, consequently, a tangential structure on $[0,a]\times Y$. After applying $\tilde D\bfB$, we obtain a tangential structure on $C=\bfB([0,a]\times Y)$ as required.

To make this strategy precise, we first note:

\begin{remark}\label{rem:k_vs_k+1}
There are compatible 1-to-1 correspondences:
\begin{enumerate} 
 \item Between collared $\cor{k+1}$-bundles $\xi$, and maps of collared $\cor k$-bundles of the form $\eps\oplus \xi'\to \xi''$, and
 \item Between collared $\cor{k+1}$-bundle maps $f\colon \xi\to \eta$ and commutative squares of collared $\cor k$-bundle maps 
\[\xymatrix{
 \eps\oplus \xi'\ar[d]^{\id\oplus f'} \ar[r] & \xi''  \ar[d]^{f''} \\
 \eps\oplus \eta' \ar[r] & \eta''
}\]
\end{enumerate}
\end{remark}

In technical, and more precise terms, the Remark states that the category $\colbun{k+1}$ is isomorphic to the category $(\eps\oplus -)/\colbun{k}$. 

\begin{proof}
Viewing the cube $\power{\underline {k+1}}$ as a product $\power{\underline k}\times [1]$, a $\cor{k+1}$-bundle $\xi$  may be viewed as a map of $\cor k$-bundle $\dell_1 \xi\to i_1^*\xi$. Now we observe that a collar $c$ on $\xi$ determines and is determined by a compatible collection of maps
\[c_{A-\{i\}, A}\colon \eps\oplus \xi(A-\{i\})\to \xi(A)\]
for $i\in A\subset \underline k$. Such a collection may be viewed, in turn, as the datum of a collar both on  $\dell_1 \xi$ and on $i_1^*\xi$, and a map
\[\eps\oplus \dell_1 \xi\to i_1^*\xi\]
of collared $\cor k$-bundles. From this description it is clear that a map $f\colon \xi\to \eta$ is collared if and only if the maps $\dell_1 f$ and $i_1^*f$ are collared, and the above square commutes.
\end{proof}

Let us  call $x$ and $y$ the first two coordinates of $\IR^{k+2}_+$ within $\rc{\infty+d}{k+2}$. Then, using Remark \ref{rem:k_vs_k+1}, the compatible tangential structures on $X$ and $W$ amount to the datum of two collared  $\cor k$-bundle maps
\[
l'_X\colon \langle \dell/\dell x\rangle\oplus i_1^*TX  \to i_1^*\theta, \quad 
l_W\colon TW \to \dell_{1} \theta, \quad
\]
such that the following diagram commutes:
\[\xymatrix{
 \langle \dell/\dell x\rangle \oplus  i_1^* TX \ar[rr]^{l'_X} 
 && i_1^* \theta
 \\
 \langle \dell/\dell x, \dell/\dell y\rangle \oplus TM \ar[r] \ar[u]
 & \langle \dell/\dell y\rangle\oplus TW  \ar[r]^(.6){\id\oplus l_W}
 & \eps\oplus \dell_1\theta. \ar[u]
}\]

Also, to endow $C:=\bfB([0,a]\times Y)$ (notation from above) with a $\theta$-structure compatible with the given structures on $X$ and $W$, it is enough to construct a collared bundle map $l'_C\colon i_1^*TC \to i_1^*\theta$, extending $l'_X$ and such that the following square is commutative:
\[\xymatrix{
i_1^*TC \ar[r]^{l'_C}
& i_1^*\theta 
\\
\langle \dell/\dell y\rangle \oplus TW \ar[u] \ar[r]^(.6){\id\oplus l_W}
& \eps\oplus \dell_1\theta. \ar[u]
}\]

We form the  push-out 
\[\xi := \langle \dell/\dell x\rangle \oplus i_1^* TX \cup  \langle \dell/\dell y\rangle \oplus  TW\]
(union of sub-bundles of $T\rc{\infty+d}{k+2}$), which is a collared $\cor k$-bundle over the $\cor k$-space $i_1^*(X)\cup W$. 

\begin{remark*}
The following may help to understand this construction: Choose an arbitrary extension $E$ of $X$ and $W$ to an (open) neat submanifold of $\rc{\infty+d}{k+2}$. If $k=0$, then the $\cor 0$- (that is, ordinary) vector bundle $\xi$ is canonically isomorphic to the restriction $TE\vert_{X\cup W}$. Here $TE$ is to be viewed as an ordinary vector bundle by forgetting the faces. In the case of general $k$, $\xi$ is canonically isomorphic to the restriction of $(i_1^* i_2^* TE)\vert_{X\cup W}$.
\end{remark*}

By construction, $\xi$ and comes with a collared $\cor k$-map $L\colon \xi\to i_1^*\theta$. Now recall that $Y=\bfB\inv(W\cup X)$ is a $\cor{k+1}$-space with first face $N$. We will next construct a collared $\cor k$-bundle map 
\[\tilde D\bfB\colon \langle \dell/\dell x\rangle \oplus i_1^* TY\to \xi.\]
Roughly, this map should be the derivative $D\bfB$ of $\bfB$. But note that the derivative of the smooth map $\bfB$ away from $p_0=(0,0)$ and $p_1=(0,a)$ does not extend to a bundle map on all of $T\rc22$. (Indeed $D\bfB$ is the identity on the positive $x$-axis and on $0\times (a,\infty)$, but rotation by $\pi/2$ on $0\times (0,a)$.) Therefore we cannot use this naive idea and need the following modification. 

Recall that by requirement (iii) for the map $\bfB$, there are neighborhoods $V_{0/1}$ of $p_{0/1}$ in $\rc22$, where $\bfB$ only scales the angle and preserves the radial coordinate. Now choose some $\eps>0$ such that $\eps$-neighborhoods $B_\eps(p_{0/1})$ around $p_{0/1}$ are contained in $V_{0/1}$, respectively.

 Choose a continuous map 
\[\rc22\to \rc22-\{p_0, p_1\}, \quad p\mapsto \tilde p\]
which is the identity except in the $\eps$-neighborhoods of $p_0$ and $p_1$. Such a map can be found by continuously retracting the respective $\eps$-neighborhoods, which are topological half-balls, onto their boundary half-circles. Then we define the map $\tilde D\bfB$ as the map of ordinary vector bundles
\[\tilde D\bfB\colon T\rc22\to T\rc22, \quad (p,v)\mapsto (\bfB(p), D_{\tilde p}\bfB(v)).\]

\begin{remark*}
 \begin{enumerate}
  \item By construction, $\tilde D \bfB$ agrees with the usual differential $D\bfB$ outside the $\eps$-neighborhoods of the non-smooth points.
  \item While domain and target of $\tilde D \bfB$ are actually $\cor 2$-vector bundles, $\tilde D \bfB$ it is not a map of $\cor 2$-vector bundles. In fact, even the map on base spaces does not extend to a map of $\cor 2$-spaces. (Recall that the main objective of $\bfB$ was precisely to change the corner structure.)
 \end{enumerate}
\end{remark*}

Since 
\[T\rc{\infty+d}{k+2} = T\rc22\times T\rc{\infty+d-2}{k},\]
we may form 
\[\tilde D\bfB:=\tilde D\bfB\times \id\colon T\rc{\infty+d}{k+2}\to T\rc{\infty+d}{k+2}.\]
Again, while domain and target of $\tilde D \bfB$ are actually $\cor {k+2}$-vector bundles, $\tilde D \bfB$ it is not a map of $\cor {k+2}$-vector bundles, but only a map of the $\cor k$-vector bundles $i_1^*i_2^*T\rc{\infty+d}{k+2}$ (that is, after forgetting the first two faces, given by the equations $x=0$ and $y=0$).

Now let us first assume that both $W$ and $X$ are $\eps$-neatly embedded at the non-smooth points of $\bfB$. Then, the collared $\cor k$-bundle map $\tilde D\bfB$ restricts to a collared $\cor k$-bundle map between sub-bundles
\[ 
\tilde D\bfB \colon \langle \dell/\dell x\rangle \oplus i_1^*TY \to  \xi.
\]
To see this, we distinguish four different cases, where we denote by $\pi\colon \rc{\infty+d}{k+2}\to \rc{2}{2}$ the projection onto the first two coordinates.
\begin{enumerate}
 \item Over $\bfB\inv (X)-\pi\inv B_\eps(p_1) \subset Y$, $\tilde D \bfB$ is the differential of $\bfB$. Therefore it sends this part of $i_1^* TY$ to $i_1^*TX$ and preserves $\langle \dell/\dell x\rangle$.
 \item Over $\bfB\inv (W)-\pi\inv(B_\eps(p_0)\cup B_\eps(p_1)) \subset Y$, $\tilde D\bfB$ is also the differential of $\bfB$. Therefore it sends this part of $i_1^* TY$ to $TW$ and identifies $\langle \dell/\dell x\rangle$ with $\langle \dell/\dell y\rangle$.
 \item Over $\pi\inv B_\eps(p_{0})$,  the $\cor k$-bundle $\langle \dell/\dell x\rangle \oplus i_1^*TY$ splits as a product $TB_\eps(p_{0})\times T(\bfB\inv N)$. The map $\tilde D\bfB$ maps this product into the product $TB_\eps(\bfB(p_0))\times TN$, which agrees with the bundle $\xi$ over $\pi\inv B_\eps(\bfB(p_0))$. 
 \item Over $\pi\inv B_\eps(p_1)$, the $\cor k$-bundle $\langle \dell/\dell x\rangle \oplus i_1^*TY$ splits as a product $TB_\eps(p_{0})\times T(\bfB\inv M)$ and we continue similarly as in (iii). 
\end{enumerate}
A similar reasoning shows that $\tilde D\bfB$ also restricts to a collared $\cor k$-bundle map 
\[\tilde D \bfB \colon i_1^*T([a,b]\times Y) \to i_1^*TC .\]
Then, we define $l'_C$ as the composite of the following collared $\cor k$-bundle maps:
\[i_1^*TC \xrightarrow{\tilde D\bfB\inv}  i_1^*T ([a,b]\times Y) \xrightarrow{\proj} \langle \dell/\dell x \rangle \oplus i_1^* TY \xrightarrow {\tilde D\bfB} \xi \xrightarrow L i_1^*\theta.\]

This finishes the construction of the $\theta$-structure on $C$. 
Applying $(\bfB, \tilde D\bfB)$ induces a homeomorphism
\[\bfB\colon \cob_\theta^\eps(Y, Z)/V \to \cob_\theta^\eps(X,Z)/(V\circ W)\]
where the addition upper index $\eps$ indicates that we only consider those morphisms that are $\eps$-neatly embedded at the non-smooth points of  $\Phi$ -- this change does not affect the homotopy type, provided $\eps$ is small enough. Then, the argument continues as above.

\section{The B\"okstedt--Madsen delooping}\label{sec:delooping}

B\"okstedt--Madsen \cite{BM} showed that a $j$-fold, \emph{nonconnective} delooping of $B\cob_d$ can be obtained as the classifying space of the $(j+1)$-tuple cobordism category $\cob_d^{j+1}$. By definition, a $j$-tuple category is defined iteratively as a (not necessarily unital) category object in the category of $(j-1)$-tuple categories. A $j$-tuple category $\calc$ comes with a $j$-fold multinerve $N^j_\bullet(\calc)$, which is a $j$-fold semi-simplicial space. The geometric realization of this $j$-fold semi-simplicial space is the \emph{classifying space} $B^j(\calc):=\vert N^j_\bullet(\calc)\vert$ of $\calc$.

For the convenience of the reader, we recollect the necessary definitions at this place. Let $\bfR$ be a category in which pull-backs exist.

(i) A (non-unital) \emph{category object} $\calc$ in $\bfR$ of two objects $\ob(\calc)$ and $\mor(\calc)$ of $\bfR$, and morphisms in $\bfR$
\[s, t\colon \mor(\calc)\to\ob(\calc), \quad \circ\colon \mor(\calc)\times_{\ob(\calc)}\mor(\calc) \to \mor(\calc)\]
satisfying the analogous conditions as the one of a (non-unital) topological category. A functor $\calc\to \cald$ of category objects in $\bfR$ is a pair of morphisms $\ob(\calc)\to \ob(\cald)$ and $\mor(\calc)\to \mor(\cald)$ in $\bfR$, which is compatible with $s$, $t$, and $\circ$. We obtain a category $\Cat(\bfR)$ of category objects in $\bfR$ and their functors. In $\Cat(\bfR)$, pull-backs exist and are given by pull-backs on objects and morphisms. A pull-back-preserving functor $F\colon \bfR\to \bfS$ induces a pull-back-preserving functor $\Cat(F)\colon \Cat(\bfR)\to \Cat(\bfS)$.  

(ii) The \emph{nerve} of a category object $\calc$ in $\bfR$ is the semi-simplicial object in $\bfR$ whose $n$-th level is the iterated pull-back
\[\mor(\calc)\times_{\ob(\calc)} \dots \times_{\ob(\calc)} \mor(\calc).\]
The nerve defines a pull-back-preserving functor
\[N_\bullet \colon \Cat(\bfR)\to s(\bfR)\]
from category objects in $\bfR$ to semi-simplicial objects in $\bfR$. The nerve is a full embedding of categories so, equivalently, a category object in $\bfR$ may be viewed as a special kind of semi-simplicial object in $\bfR$.

The nerve functor gives rise to a pull-back-preserving functor from $j$-tuple categories to $j$-fold semi-simplicial spaces, the $j$-fold multinerve: It is defined iteratively as the composite
\begin{multline*}
 N^j \colon \Cat^j(\Top) = \Cat(\Cat^{j-1}\Top) \xrightarrow{\Cat(N^{j-1})} \Cat(s^{j-1}(\Top)) \\
 \xrightarrow{N} s(s^{j-1}(\Top)) = s^j(\Top)
\end{multline*}
where  the second functor is the nerve for category objects in $s^{j-1}(\Top)$. 

The multinerve is also a full embedding of categories. Therefore, a $j$-tuple category  may be seen as a $j$-fold semi-simplicial space $\calc_{\bullet, \dots, \bullet}$ such that each ``slice'' obtained by evaluating $\calc$ at $(j-1)$ numbers at arbitrary places, the resulting semi-simplicial space is (the nerve of) a category. 
%Taking classifying space in each simplicial direction defines a space $B^j\calc$, the classifying space of the $j$-tuple category $\calc$. 

The $d$-dimensional $j$-tuple cobordism category of $\cor k$-manifolds $\cob_{d, \cor k}^j$ is defined as follows: Let $\vec p=(p_1, \dots, p_j)$, $\vec p \leq \vec 1$, and denote by $\ell$ the number of 0-entries of $\vec p$. An element of $(\cob_d^j)_{\vec p}$ is defined to be a collection of real numbers $a_1, \dots, a_j$, such that for all $n\in \{1,\dots, j\}$ we have
\[a_n =0 \quad \mathrm{if~}p_n=0, \qquad a_n >0 \quad \mathrm{if~}p_n=1,\]
together with a subset 
\[M\subset \rc{\infty+d-j}{k}\times \prod_{n=0}^j [0, a_n] \]
which is a neatly embedded compact $(d-\ell)$-dimensional submanifold. The differentials $d_0$ and $d_1$, in the $n$-th simplicial direction, are given by restricting to $0$ and $a_n$, respectively, in the $n$-th coordinate of the product, and by setting the value of $a_n$ to be 0. Composition in the $n$-th direction is given by stacking in the $n$-th coordinate and adding the $n$-th real numbers.

Again, instead of defining a topology, we view this as the set of 0-simplices of a simplicial object in $j$-tuple categories, and geometrically realize this simplicial object. The details are just as above.

\begin{theorem}\label{thm:delooping}
For $j\geq 2$, there is a canonical a weak equivalence
\[B^{j-1}\cob_{d, \cor k}^{j-1} \simeq \Omega B^j\cob^j_{d, \cor  k}.\] 
\end{theorem}

The analogous statement holds for the the obvious generalization of $\cob_{d, \cor k}$ incorporating $\theta$-structures, where $\theta$ is a $d$-dimensional collared $\cor k$-vector bundle. Details are left to the reader.

Theorem \ref{thm:delooping} follows by combining three Lemmas which we are about to state. We abbreviate by $\cob^j$ the $j$-fold cobordism category of $\cor k$-manifolds in dimension $d$. Denote by $B^i \cob^j$ the geometric realization of $\cob^j$ in the first $i$ semi-simplicial directions. It is a $(j-i)$-fold semi-simplicial space. Then, we have:

\begin{lemma}\label{lem:delooping}
For $j\geq 1$, the semi-simplicial space $B^{j-1} \cob^j$ is a weakly unital semi-Segal space. If $j\geq 2$, it is a groupoid, \emph{i.e.}, every 1-simplex is an equivalence. 
\end{lemma}

For a semi-Segal space $\calc$ and some object $c\in\calc_0$, we may define the \emph{endomorphism space} $\End_\calc(c)$ of $c$ as the  homotopy fiber of
\[(d_1, d_0)\colon \calc_1\to \calc_0\times \calc_0\]
over $(c,c)$. 

\begin{lemma}\label{lem:groupoid_group_completion}
If $\calc$ is a weakly unital semi-Segal groupoid, then for any $c\in \calc_0$ the canonical map 
\[\End_\calc(c)\to \Omega_c B\calc\]
is an equivalence.  
\end{lemma}

Note that the Lemma is wrong without the hypothesis of weak unitality, as the example a space viewed as a discrete category groupoid shows. Finally, we have:

\begin{lemma}\label{lem:endomorphism_of_B_i-1_cob}
For $j\geq 2$, the canonical map
\[B^{j-1}\cob^{j-1}\to \End_{B^{j-1}\cob^j}(\emptyset) \]
is an equivalence.  
\end{lemma}

The remainder of this section is devoted to the proof of the three Lemmas. It will be convenient to use the notion of a 
\emph{local fibration} $P\colon \calc\to \cald$ between $j$-tuple categories. It is defined
inductively on $j$ as follows: For $j=1$, this is a local fibration between categories in the sense defined in section \ref{sec:local_to_global}. For $j>1$, $P$ defined to be a local fibration if the maps
\[
\begin{split}
 P& \colon \calc_{0, \bullet}\to \cald_{0, \bullet}, \\
 (P, d_1, d_0)&\colon \calc_{1, \bullet} \to \cald_{1, \bullet}\times_{\cald_{0, \bullet}\times \cald_{0,\bullet}} (\calc_{0, \bullet}\times \calc_{0, \bullet})   
\end{split}
\]
are local fibrations of $(j-1)$-tuple categories. A $j$-tuple category is \emph{locally fibrant} if the projection to the terminal object is a local fibration.  

\begin{lemma}\label{lem:multi_local_fibrations} Let $P\colon \calc\to \cald$ be a functor of $j$-tuple categories. 
\begin{enumerate}
 \item If $P$ is a local fibration, then it is a level fibration, that is, for each $\vec p\in \Delta_{inj}^j$, the map of spaces $\calc_{\vec p}\to \cald_{\vec p}$ is a fibration.
 \item $P$ is a local fibration if and only if for each $\vec p\leq \vec 1=(1,\dots, 1)\in \Delta_{inj}^j$, the canonical map
 \[\calc_{\vec p}
   \to 
   \cald_{\vec p}
      \underset
      {
         \underset
           {\substack{\vec q\to \vec p\\\vec q\neq \vec p}}
           {\lim}
         \cald_{\vec q}
      }
      {\times}
   \lim_{\substack{\vec q\to \vec p\\ \vec q\neq \vec p}} \calc_{\vec q}
  \]
  is a fibration. In particular, the notion of local fibration does not depend on the ordering of indices. 
  \item Suppose that $\calc$ is a locally fibrant $j$-tuple category. On each $\calc_{\vec p}$, the first and the last differentials in each of the simplicial directions are fibrations. 
\end{enumerate}
\end{lemma}

\begin{proof}
(i) is an easy induction on $j$. Let us prove (ii), which is also an induction on $j$. The case $j=1$ is true by definition. For $j>1$, we write $\vec p=(i, \vec k)$ where either $i=0$ or $i=1$. Then by induction hypothesis, $\calc_{0, \bullet}\to \cald_{0, \bullet}$ is a local fibration if and only if for each $\vec k\leq \vec 1$, the map
\[
\calc_{0, \vec k} \to \cald_{0, \vec k} 
\underset{\lim_{\vec \ell}\cald_{0, \vec \ell}}{\times} 
\lim_{\vec \ell} \calc_{0, \vec \ell}
=
\cald_{\vec p} 
\underset{\lim_{\vec q} \cald_{\vec q}}{\times }
\lim_{\vec q} \calc_{\vec q}
\]
is a fibration. (For notational brevity we omit here and in the following the specification of the indices $\vec \ell$ and $\vec q$, which is just as above.) On the other hand the map $(P, d_1, d_0)$ is a local fibration if and only if for each $\vec k \leq \vec 1$, the map
\[
\calc_{1, \vec k} 
\to 
\bigl(
\cald_{1, \vec k} 
\underset{(\cald_{0, \vec k}\times \cald_{0, \vec k})}{\times}
(\calc_{0, \vec k} \times \calc_{0, \vec k}) 
\bigr)
\underset
{\lim_{\vec \ell} 
\bigl(
  \cald_{1, \vec \ell}
  \underset{(\cald_{0, \vec \ell}\times \cald_{0, \vec \ell})}{\times}
  (\calc_{0, \vec \ell}\times \calc_{0, \vec \ell})
\bigr)}
{\times}
\lim_{\vec \ell} \calc_{1, \vec \ell}
\]
is a fibration. Let us write the right factor as $\lim(\cald_{1, \vec \ell}\times_{\cald_{1, \vec \ell}} \calc_{1, \vec \ell})$, distribute the limit over the fiber product, and switch the ordering in which the fiber products are taken. Then the target takes the form
\[
\cald_{1, \vec k} 
\underset
{\bigl(
(\cald_{0, \vec k}\times \cald_{0, \vec k})
\underset{\lim_{\vec \ell} (\cald_{0, \vec \ell}\times \cald_{0, \vec \ell})}{\times}
\lim_{\vec \ell} \cald_{1, \vec \ell}
\bigr)}
{\times}
\bigl(
(\calc_{0, \vec k}\times \calc_{0, \vec k})
\underset{\lim_{\vec \ell} (\calc_{0, \vec \ell}\times \calc_{0, \vec \ell})}{\times}
\lim_{\vec \ell} \calc_{1, \vec \ell}
\bigr)
\]
and this is equivalent to
\[ 
 \cald_{\vec p} 
    \underset{\lim_{\vec q} \cald_{\vec q}}{\times}
  \lim_{\vec q} \calc_{\vec q}.
\]
Therefore $P$ is a local fibration if and only if the required map is a fibration, for both choices of $i$ and all choices of $\vec k\leq \vec 1$, that is, for all choices of $\vec q\leq \vec 1$.

To see (iii), we remark from (i) that both maps $d_0$ and $d_1$ are level fibrations; a priori only in the first simplicial direction, but then also in each other directions by part (ii). But any of the maps from the statement is obtained from one of these maps by pull-back.
\end{proof}

\begin{lemma}\label{lem:multicob_locally_fibrant}
The $j$-fold category $\cob^j$ is locally fibrant. 
\end{lemma}

\begin{proof}
The map
\[\cob^j_{\vec 1} \to \lim_{\substack{\vec q\to \vec 1 \\ \vec q\neq \vec 1}} \cob^j_{\vec q}\]
is a fibration; for $j=1$ this is the fact that the (1-tuple) cobordism category is locally fibrant, but the argument from section \ref{sec:genauer}, using Theorem \ref{thm:isotopy_extension}, generalizes to the case of general $j$. 

This already shows that criterion (ii) from Lemma \ref{lem:multi_local_fibrations} holds vor $\vec p=\vec 1$. But for general $\vec p\leq \vec 1$ we have an isomorphism  $\cob^j_{\vec p}\cong \cob^{j-i}_{\vec 1}$, with $i$ the number of $0$'s in $\vec p$. Therefore the criterion (ii) from Lemma \ref{lem:multi_local_fibrations} applies to all $\vec p \leq \vec 1$. 
\end{proof}

We move on to the proof of Lemma \ref{lem:delooping}. To this end, we formulate a more general statement. 

\begin{lemma}\label{lem:delooping_general}
Let $0\leq i < j$. For any $\vec q \in \Delta_{inj}^{j-i-1}$, we have:
\begin{enumerate}
 \item $B^i \cob^j_{\bullet, \vec q}$ is a semi-Segal space.
 
%   \item The map
%  \[-\times [0,a]\colon B^i\cob^j_{\bullet, 0, \vec p} \to B^i\cob^j_{\bullet, 1, \vec p},\]
%  induced by taking product with the interval $[0,a]$, preserves equivalences.
 
  \item For $i>0$, the 0-skeleton inclusion 
 \[B^{i-1}\cob^j_{0, \bullet, \vec q} \to B^i \cob^j_{\bullet, \vec q}\]
 preserves equivalences.
 
 \item $B^i \cob^j_{\bullet, \vec q}$ is weakly unital and, for $i>0$, even group-like.
\end{enumerate}
If, futhermore, $j>i+1$, then for any  $\vec p\in \Delta_{inj}^{j-i-2}$, we have:
\begin{enumerate}
\setcounter{enumi}{3}
 \item The face maps
 \[d_{0/1}\colon B^i\cob^j_{\bullet, 1, \vec p} \to B^i\cob^j_{\bullet, 0, \vec p}\]
 are weakly unital and cartesian and cocartesian fibrations.
\end{enumerate}
\end{lemma}

\begin{proof}
The proof is by induction on $i$. We first let $i=0$. Then, $\cob^j_{\bullet, \vec q}$ is a locally fibrant topological category, and therefore a semi-Segal space. This shows (i); (ii) is void for $i=0$. Statement (iii) is immediate since cylinders are weak units. As for (iv), the face maps are weakly unital because they preserve cylinders. It remains to show the last claim of (iv). To this end, we view a cobordism between $\cor k$-manifolds as a $\cor{k+2}$-manifold whose last two faces do not meet --- this is done by identifying $\rc{\infty+d-1}{k}\times [0,1]$ with $\rc{\infty+d}{k+2}-\dell_{k+2}\dell_{k+1} \rc{\infty+d}{k+2}$. (If a tangential structure is present, this also involves the construction of a suitable collared $\cor{k+2}$-bundle on the $\cor{k+2}$-manifold, by means of Remark \ref{rem:k_vs_k+1}).  It follows that we have a homotopy pull-back square
\begin{equation}\label{eq:d_vs_dell}
 \xymatrix{
 \cob^j_{\bullet, 1, \vec q} \ar[d]^{d_{0/1}} \ar[rr] && (\cob^{j-1}_{d,\cor{k+2}})_{\bullet, \vec q} \ar[d]^{\dell_{k+1/k+2}}\\
 \cob^j_{\bullet, 0, \vec q} \ar[rr]^{i_{k+2/k+1}}  && (\cob^{j-1}_{d,\cor{k+1}})_{\bullet, \vec q}
}
\end{equation}
and the right vertical map is a cartesian and cocartesian fibration --- if $\vec q=0$ this is Lemma \ref{lem:boundary_map_cocartesian_fibration}, and the proof carries over verbatim to the general case. By Corollary \ref{cor:cocartesian_fibration_stable_under_pullback} the left map is then also a cartesian and cocartesian fibration.  

We move on to the inductive step and assume that the statements are true for some $i-1\geq 0$; let us now prove them for $i$. 

(i) By inductive hypothesis, each $B^{i-1}\cob^j_{r, \bullet, \vec q}$ is a semi-Segal space. Therefore, we have
\[B^i \cob^j_{n+1, \vec q} \simeq B\bigl(B^{i-1}\cob^j_{\bullet, n, \vec q}\times_{B^{i-1}\cob^j_{\bullet, 0, \vec q}}^h B^{i-1}\cob^j_{\bullet, 1, \vec q}\bigr).\]
By inductive hypothesis, the assumptions from the local-to-global principle, Theorem \ref{thm:main_Segal} hold. We conclude that 
\[ B^i \cob^j_{n+1, \vec q} \simeq B^i \cob^j_{n, \vec q} \times^h_{B^i \cob^j_{0, \vec q}} B^i \cob^j_{1, \vec q}\]
and it follows that the Segal condition holds indeed.

(ii) Let $f\in B^{i-1}\cob^j_{0, 1, \vec q}$ be an equivalence with target $x\in B^{i-1}\cob^j_{0,0,\vec q}$. It is enough to show that its image $f'$ in $B^i \cob^j_{1, \vec q}$ is cartesian: By self-duality $f'$ will then also be cocartesian, hence an equivalence.

Thus we need to show that the map
\begin{equation}\label{eq:weak_units_in_higher_cobcats}
B^i \cob^j_{2,\vec q}\times^h_{B^i \cob^j_{1,\vec q}} \{f'\}
 \to
 B^i \cob^j_{1,\vec q}\times^h_{B^i \cob^j_{0,\vec q}} \{x'\}
\end{equation}
is an equivalence. Now the inclusion of $\{f'\}$ into $B^i \cob^j_{1, \vec q}$ may be, up to homotopy equivalence, rewritten as the geometric realization of the semi-simplicial map
\[N_\bullet [0, \infty)\times \{f'\} \to B^{i} \cob^j_{\bullet, 1,\vec q}\]
that embeds cylinders on $f$ of variable length. Using this and the analogous statement for $x$',  \eqref{eq:weak_units_in_higher_cobcats} may be rewritten as the map
\begin{multline*}
 B^i \cob^j_{2,\vec q}\times^h_{B^i \cob^j_{1,\vec q}} B(N_\bullet [0, \infty)\times \{f'\})\\
 \to
 B^i \cob^j_{1,\vec q}\times^h_{B^i \cob^j_{0,\vec q}} B(N_\bullet [0, \infty) \times \{x'\}).
\end{multline*}

Applying the local-to-global principle, the target of this map is equivalent to the geometric realization of
\[
 B^{i-1} \cob^j_{\bullet, 1,\vec q}\times^h_{B^{i-1} \cob^j_{\bullet, 0,\vec q}} (N_\bullet [0, \infty) \times \{x'\}).
\]
Similarly, the domain is equivalent to the geometric realization of
\[
 B^{i-1} \cob^j_{\bullet, 2,\vec q}\times^h_{B^{i-1} \cob^j_{\bullet, 1,\vec q}} (N_\bullet [0, \infty)\times \{f'\}).
\]
Thus, it is enough to show that the canonical map of semi-simplicial spaces
\begin{multline*}
 B^{i-1} \cob^j_{\bullet, 2,\vec q}\times^h_{B^{i-1} \cob^j_{\bullet, 1,\vec q}} (N_\bullet [0, \infty)\times \{f'\})\\
 \to
 B^{i-1} \cob^j_{\bullet, 1,\vec q}\times^h_{B^{i-1} \cob^j_{\bullet, 0,\vec q}} (N_\bullet [0, \infty) \times \{x'\})
\end{multline*}
induces an equivalence on classifying spaces. We show this by proving that it is indeed a level equivalence.

By the semi-Segal condition, it is enough to consider the cases $\bullet=0$ and $\bullet=1$. For $\bullet=0$, the statement holds since $f$ is an equivalence by assumption. For $\bullet=1$ we need to show that $f\times [0,a]$ is also an equivalence. If $i>1$ then $i-1>0$, so $f\times [0,a]\in B^{i-1}\cob^j_{1,1,\vec q}$ is an equivalence by inductive hypothesis (iv). If $i-1=0$ then we have to show that the functor 
\[-\times [0,a]\colon \cob^j_{0, \bullet, \vec q} \to \cob^j_{1, \bullet,  \vec q}\]
is weakly unital. But this is clear since weak units in this case are cylinders, up to homotopy, and these are clearly preserved under taking product with $[0,a]$. 

(iii) Since the 0-skeleton inclusion is $\pi_0$-surjective, it is enough to show that the image of the 0-skeleton inclusion consists of equivalences. If $i>1$, this follows from (ii) and the inductive hypothesis of (iii); we are left to consider the case $i=1$. So let $f\in \cob^j_{0, 1, \vec q}$ be a cobordism, and $f'$ its reverse cobordism (obtained by flipping the cobordism axis). Then the composite cobordisms 
\[f' \circ f, f\circ f' \in \cob^j_{0, 1, \vec q}\]
are cobordant to cylinders. So their images under the 0-skeleton inclusion are homotopic to the images of cylinders.

By part (ii), we conclude that the images of $f'\circ f$ and $f\circ f'$ under the 0-skeleton inclusion are equivalences. We deduce that both $f'$ and $f$ are also mapped to equivalences.

(iv) In view of (iii), weak unitality of the face maps is automatic. Also any morphism is $d_{0/1}$-cartesian and cocartesian, so for the second statement it suffices to show that morphisms can be lifted at all, \emph{i.e.,} the four maps
\[(d_{0/1}, d_{0/1})\colon B^{i}\cob^j_{1, 1,\vec p}\to B^{i}\cob^j_{1, 0,\vec p} \times^h_{B^i\cob^j_{0, 0,\vec p}} B^{i}\cob^j_{0, 1,\vec p}\]
are $\pi_0$-surjective. But again by the local-to-global principle, these are equivalent to the realizations of maps
\[B^{i-1}\cob^j_{\bullet,1, 1,\vec p}\to B^{i-1}\cob^j_{\bullet,1, 0,\vec p} \times^h_{B^{i-1}\cob^j_{\bullet,0, 0,\vec p}} B^{i-1}\cob^j_{\bullet,0, 1,\vec p}.\]
Now, $\pi_0$-surjectivity on the geometric realizations follows from $\pi_0$-surjectivity on the level of 0-simplices. But this latter statement follows from the fact that the face maps
\[d_{0/1}\colon B^{i-1}\cob^j_{0,\bullet, 1, \vec p} \to  B^{i-1}\cob^j_{0,\bullet, 0, \vec p}\]
are cartesian-cocartesian fibrations, by inductive hypothesis.  
\end{proof}

\begin{proof}[Proof of Lemma \ref{lem:endomorphism_of_B_i-1_cob}]
For each $0\leq i<j$ and each $(j-i-1)$-multi-index $\vec p$, there is a more or  less canonical map
\[B^i \cob^{j-1}_{\vec p} \to B^i \cob^j_{1, \vec p}\]
induced by the inclusion of one $\IR$-coordinate into $[0,1]$. Its composition with $d_{0/1}$ takes values in the contractible semi-simplicial subspace $B^i \cob^j_{0, \vec p}$ on configurations involving only the empty set, and therefore defines a map 
\[B^i\cob^{j-1}_{\vec p}\to \End_{B^i\cob^{j-1}_{\vec p}}(\emptyset).\]
We claim that this map is an equivalence, for each $i$ and $\vec p$. 

First, a nullbordism of $\cor k$-manifolds may be viewed as a $\cor{k+1}$-manifold, by forgetting one length coordinate.  It follows that there is a commutative square
\[\xymatrix{
 (B^{i} \cob_{d, \cor{k+1}}^{j-1})_{\vec p} \ar[d] \ar[r]
 & (B^{i}\cob^j)_{1, \vec p} \ar[d]^{d_1}
 \\
 B^i (N_{\bullet} (0, \infty)^{j-1})_{\vec p} \ar[r]^{\emptyset}
 & (B^{i}\cob^j)_{0, \vec p}
}\]
where the  lower horizontal map includes cylinders at the empty set. It is a level-wise homotopy pull-back for $i=0$, by local fibrancy of $\cob^j$. 

Let us know view this as a diagram of semi-simplicial spaces in the first entry of $\vec p$. It was shown in the proof of Lemma \ref{lem:delooping} that for any $i\geq 0$, the right vertical map is a cartesian and cocartesian fibration between weakly unital semi-Segal spaces (groupoids, for $i>0$).  Therefore, applying the local-to-global principle repeatedly, it follows that the square is a homotopy pull-back for any $i\leq j-1$. Since the lower left corner is contractible, we obtain a fibration sequence
\begin{equation}\label{eq:first_fibration_sequence}
(B^{i} \cob_{d, \cor{k+1}}^{j-1})_{\vec p} \to (B^{i}\cob^j)_{1, \vec p} \xrightarrow{d_0} B^{i}\cob_{0, \vec p}^j
\end{equation}

Next, there is a commutative square
\[\xymatrix{
 (B^i \cob^{j-1})_{\vec p} \ar[r] \ar[d]
 & (B^{i} \cob_{d, \cor{k+1}}^{j-1})_{\vec p} \ar[d]^{\dell_{k+1}}
 \\
 B^i(N_{\bullet} (0, \infty)^{i-1})_{\vec p} \ar[r]^{\emptyset}
 & (B^i \cob^{j-1})_{\vec p}
}\]
It is also a homotopy pull-back for $i=0$. 

Let us again view this as a diagram of weakly unital semi-Segal spaces (groupoids, if $i>0$) in the first entry of $\vec p$. The right vertical map is a cartesian and cocartesian fibration for $i=0$. Continuing to argue iteratively as in the proof of Lemma \ref{lem:delooping}, it follows that for $i\leq j-1$, the right vertical map continues to be a cartesian and cocartesian fibration, and the diagram continues to be a homotopy pull-back. Therefore we obtain a fibration sequence
\begin{equation}\label{eq:second_fibration_sequence}
(B^i \cob^{j-1})_{\vec p} \to (B^{i} \cob_{d, \cor{k+1}}^{j-1})_{\vec p} \xrightarrow{\dell_{k+1}} (B^{i}\cob^j)_{0, \vec p}
\end{equation}

Since the diagram
\[\xymatrix{
 (B^{i} \cob_{d, \cor{k+1}}^{j-1})_{\vec p} \ar[rr] \ar[rd]_{\dell_{k+1}}
 && B^{i}\cob_{1, \vec p}^j \ar[ld]^{d_1}
 \\
 & B^{i}\cob_{0, \vec p}^j
}\]
commutes, we conclude from \eqref{eq:first_fibration_sequence} and \eqref{eq:second_fibration_sequence} the existence of the desired fibration sequence
\[B^i \cob^{j-1}_{\vec p} \to B^{i}\cob_{1, \vec p}^j \xrightarrow{(d_0, d_1)} B^{i}\cob_{0, \vec p}^j \times B^{i}\cob_{0, \vec p}^j.\qedhere\]
\end{proof}

It remains to give the

\begin{proof}[Proof of Lemma \ref{lem:groupoid_group_completion}]
We may assume that $\calc$ is Reedy fibrant. We denote by $c\wr \calc$ the pull-back semi-simplicial space
\[\xymatrix{
 c\wr\calc \ar[r] \ar[d] & \calc_{1+\bullet} \ar[d]^f\\
 \{c\} \ar[r] & \calc_0
}\]
where $f$ is the first-vertex map. The canonical projection $c\wr\calc\to \calc$, induced by $d_0$ on $\calc_{1+\bullet}$, comes with a canonical semi-simplicial nullhomotopy $(c\wr\calc)_n\to \calc_{n+1}$, induced by the identity on $\calc_{1+\bullet}$.  Then, following the proof of \cite[3.1]{RS_Hcob}, we see that
\[\End_\calc(c) \simeq \Omega_c B\calc \times B(c\wr\calc)\]
and it remains to show that the second factor is contractible. 

Since $\calc$ is weakly unital, there is a cocartesian edge $f\in \calc_1$ starting at $c$. Define $f\wr\calc$ as the pull-back
\[\xymatrix{
 f\wr\calc \ar[r] \ar[d] & \calc_{2+\bullet} \ar[d]^f\\
 \{f\} \ar[r] & \calc_1
}\]
where $f$ takes the first 1-simplex. The canonical projection $f\wr\calc\to c\wr\calc$, induced by $d_1$, comes with a canonical semi-simplicial nullhomotopy, induced by the identity on $\calc_{2+n}$. On the other hand, it is also an equivalence in semi-simplicial level 0, in view of the fact that $f$ is cocartesian; and it follows from the Segal condition that it is an equivalence in each semi-simplicial degree. Therefore, the induced map on classifying spaces is both nullhomotopic and an equivalence.
\end{proof}

\appendix

\section{Isotopy extension on \texorpdfstring{$\cor{k}$}{<k>}-manifolds}\label{sec:isotopy_extension}

Recall that an $n$-dimensional (smooth) manifold is a second countable paracompact topological space equipped with an atlas of local homeomorphisms to open subsets of $\IR^n_+=[0,\infty)^n$, such that the change of charts is smooth. (This means that our definition of manifold allows corners.) A \emph{submanifold} of $M$ is a subset $N\subset M$ which, locally in a suitable chart, looks like $\IR^m_+\times \{1\} \subset \IR^n_+$. 
%A \emph{smooth embedding} $K\to M$ is a smooth map inducing a diffeomorphism onto a smooth submanifold.

All manifolds in this appendix will be smooth. One major example of a manifold with corners is $\rc{n}{k}:=\IR^{n-k}\times \IR^k_+$. We note that this comes with the extra structure of a $(k+1)$-ad in the sense of Wall: That is, its boundary (as a topological manifold)  is partitioned into the $k$ many subspaces
\[\dell_i \rc{n}{k} := \{x\in \rc{n}{k}\;\vert\; x_{n-k+i}=0\} \quad (i\in \{1,\dots, k\}).\]
(This $(k+1)$-ad structure is compatible with the manifold  structure  in a suitable way, giving $\rc{n}{k}$ the structure of a $\cor{k}$-manifold, see \cite{Laures} or \cite{Genauer}). We denote as usual, for $A\subset \underline k = \{1,\dots, k\}$, the value of the $(k+1)$-ad $\rc{n}{k}$ at $A$, given by 
\[\rc{n}{k}(A):=\bigcap_{i\notin A} \dell_i \rc{n}{k}=\{x\in \rc{n}{k}\;\vert\; \forall_{i\notin A} x_{n-k+i}=0\}.\]

For $A\subset B\subset \underline k$, we identify $\rc n k(B)$ with $\rc n k (A)\times [0,\infty)^{B - A}$ by writing the coordinates $x_{n-k+i}$, $i\in B\setminus A$, into the second factor. With this identification, we have canonical \emph{collar embeddings}
\[c_{AB}\colon \rc{n}{k}(A) \times [0,\eps)^{B-A} \to \rc{n}{k}(B),\]
for $A\subset B\subset \underline k$, and any $\eps>0$.

The only manifolds that we will consider are submanifolds $M\subset \rc{n}{k}$. These inherit the structure of a $(k+1)$-ad by means of 
\[M(A):= M\cap \rc{n}{k}(A).\]
(In fact, they inherit the structure of a $\cor k$-manifolds; conversely any $\cor k$-manifold is isomorphic, as manifold and as $(k+1)$-ad, to a submanifold $M\subset \rc n k$.)

Furthermore we will restrict our attention to submanifolds which are \emph{neat}: For all $A\subset B \subset \underline k$, we require that there is an $\eps>0$ such that 
 \[M(B) \cap (\rc{n}{k}(A) \times [0,\eps)^{B-A})= M(A) \times [0,\eps)^{B-A}.\]
If we want to specify the value of $\eps$, we also speak of an \emph{$\eps$-neat submanifold}.
An $\eps$-neat smooth submanifold $M\subset \rc{n}{k}$ inherits extra structure of collar embeddings
\[c_{AB}\colon M(A) \times [0,\eps)^{B-A} \to M(B),\]
by restricting the collar embeddings on $\rc n k$. These are compatible in the sense that  for each $A\subset B\subset C\subset \underline k$, the following triangle is commutative:
\[\xymatrix{
 M(A)\times [0,\eps)^{C-A} \ar[rr]^{c_{AB}\times \id_{[0,\eps)^{C-B}}} \ar[rrd]_{c_{AC}}
 && M(B)\times [0,\eps)^{C-B} \ar[d]^{c_{BC}}
 \\
 && M(C)
}\]

Let $M, N\subset  \rc{n}{k}$ be $\eps$-neat submanifolds. By definition, an \emph{$\eps$-neat embedding,} resp.~\emph{$\eps$-neat diffeomorphism} $e\colon M\to N$ is a map which is both an allowable map of $(k+1)$-ads (that is, for all $A\subset \underline k$, we have $e\inv(N(A))=M(A)$) and an embedding (resp.~diffeomorphism) of manifolds, which is cylindrical at the collars in the sense that for each $A\subset B\subset \underline k$, the following square  commutes:
\[\xymatrix{
 M(A) \times [0,\eps)^{B - A} \ar[rr]^{c_{AB}} \ar[d]^{e\times \id}
 && M(B)  \ar[d]^{e}
\\
N(A)\times [0,\eps)^{B - A} \ar[rr]^{c_{AB}}
 && N(B) 
}\]

We will now study family versions of isotopy extension theorems for neat submanifolds of $\rc{n}{k}$. To this end, we let
\[\Emb^\eps(M,\rc n k) \quad \mathrm{and} \quad \Aut^\eps(M)\]
denote the spaces of $\eps$-neat embeddings $M\to \rc n k$ and $\eps$-neat automorphisms $M\to M$, equipped with the strong $C^r$-topology \cite[I.4.3]{Cerf}, $r\geq 1$. We then define
\[\Emb(M,\rc n k)=\colim_{\eps} \Emb^\eps(M,\rc n k) \quad \mathrm{and}\quad \Aut(M)=\colim_{\eps} \Aut^\eps(M)\]
equipped with the colimit topology.

We are interested in restricting an embedding $M\to \rc n k$ to an embedding $M(A)\to \rc n k(A)$ for $A\subset \underline k$, or to an embedding $\dell M\to \dell \rc n k$. As $\dell M= \cup_{A\subsetneq \underline k} M(A)$, we \emph{define}
\[\Emb^\eps(\dell M,\dell \rc n k):=\lim_{A\subsetneq \underline k} \Emb^\eps(M(A), \rc n k(A)),\]
and, more generally, for a subcomplex $X\subset \Delta^{\underline k}$ (that is, a non-empty subset $X\subset \power{\underline k}$ which is closed under taking subsets),
\[\Emb^\eps(M(X), \rc n k(X)):=\lim_{A\in X} \Emb^\eps(M(A), \rc n k(A)).\]
Note that
\[\Emb^\eps(M(A), \rc n k(A)) \cong \Emb^\eps(M(\sigma_A), \rc n k(\sigma_A)\]
with $\sigma_A := \{B\subset \underline k\;\vert\; B\subset A\}$ the simplex spanned by $A$, so the first case is also included in this notation. Again we let
\begin{align*}
\Emb(M(X), \rc n k(X)) &:= \colim_\eps \Emb^\eps(M(X), \rc n k(X)), \\
\Aut(M(X))  &:= \colim_\eps \Aut^\eps(M(X)).
\end{align*}

\begin{theorem}\label{thm:isotopy_extension}
Let $M\subset \rc n k$ be a compact neat submanifold. Let further $Y\subset X\subset \Delta^{\underline k}$ be subcomplexes. In the square formed by restriction maps
\[\xymatrix{
\Aut(\rc n k(X)) \ar[rr] \ar[d] && \Emb(M(X), \rc n k(X)) \ar[d]\\
\Aut(\rc n k(Y)) \ar[rr] && \Emb(M(Y), \rc n k(Y))
}\]
all maps are Serre fibrations, as well as the map from the left corner to the pull-back of the remaining diagram.
\end{theorem}

\begin{proof}
\emph{Step 1.} We show that the left vertical map admits a local section
\[s\colon \Aut(\rc n k(Y))\supset U\to \Aut(\rc n k(X))\]
on an open neighborhood $U$ of the neutral element, so that  the map under consideration is locally trivial. We first construct, for any $i\in \underline k$, a local section of the map
$\Aut(\rc n k)\to \Aut(\dell_i \rc n k)$.

Choose some smooth map 
\[\alpha\colon [0,\infty)\to [0,1]\]
which is identically 0 near 0 and identically 1 on $[1,\infty)$, and consider the map
\[\hat s_i\colon \Aut(\dell_i \rc n k) \to  C^r(\rc n k, \rc n k)\]
defined by 
\[\hat s_i(\phi)(x,t) :=  \alpha(t)\cdot x + (1-\alpha(t))\cdot \phi(x)\]
where we identify $\rc n k$ with $\dell_i \rc n k \times [0,\infty)$. Since $\Aut(\rc n k)$ is open in the space of all neat and allowable $C^r$-selfmaps of $\rc n k$ (this  follows from \cite[I.1.4.2]{Cerf}), the map $\hat s_i$ restricts to a local section $s_i$ as required.  

Now we show that the map $\alpha\colon \Aut(\rc n k)\to \Aut(\dell \rc n k)$ also has a local section. To this end, we note that if $x$ is a fixed point of $\phi\in \Aut(\dell_i\rc n k)$, then so is it a fixed point of $\hat s_i(\phi)$. Therefore, if we locally define maps 
\[\sigma_i\colon \Aut(\dell\rc n k) \supset U_i\to \Aut(\rc n k), \quad i\in \underline k,
 \]
iteratively by
\[\sigma_1:=s_1,\quad  \sigma_{i+1}(\phi):=\sigma_i(\phi)\cdot s_{i+1}(\sigma_i(\phi)\inv\cdot \phi),\]
then we iteratively see that $\sigma_i(\phi)$ restricts to $\phi$ over the first $i$ faces, so that $\sigma_k$ defines a local section as required.  

Since, for any $A\subset \underline k$, we have $\rc n k (A) = \rc{n-k+\vert A\vert}{\vert A\vert}$, the above shows that the restriction map $\Aut(\rc n k(A))=\Aut(\rc n k(\sigma_A))\to \Aut(\rc n k(\dell\sigma_A)$ has a local section for all $\emptyset \neq A\subset \underline k$. But now, in the general case, the inclusion $X\to Y$ may be factored into inclusions each of which fills in precisely one simplex, and therefore has a local section. Composing these local sections defines a local section $s$ as required in the general case.

\emph{Step 2.} We show that the map from the left upper corner to the pull-back of the remaining diagram admits a local section. As in step 1, one reduces to the case where $X$ is the simplex spanned by $\underline k$, and $Y$ its boundary; so we need to show that the map
\[\Aut(\rc n k) \to \Emb(M, \rc n k)\times_{\Emb(\dell M, \dell \rc n k)} \Aut(\dell \rc n k)\]
has a local section. Using the local section from step 1, one reduces to showing that the restriction map 
\[\Aut(\rc n k; \dell) \to  \Emb(M, \rc n k; \dell)\]
on diffeomorphism group and embedding space relative boundary has a local section. As the behavior near the boundary is standard, we can smooth corners out. But then the argument from \cite{Lima} guarantees the existence of a local section.

\emph{Step 3.}  Taking $Y=\emptyset$ in Step 2, we conclude that the upper (hence also the lower) horizontal map is a Serre fibration; as well as the diagonal map. But then it easily follows that the right vertical map is also a Serre fibration. %Again we only need to consider the case of the map
% \[\Emb(M, \rc n k) \to \Emb(\dell M, \dell \rc n k).\]
% 
% By our results so far, the map $\Aut(\rc n k)\to \Emb(\dell M, \dell \rc n k)$ is a Serre fibration. It follows formally that $\Emb(M, \rc n k)\to \Emb(\dell M, \dell \rc n k)$ is a Serre fibration on each component of the total space in the image of $\Aut(\rc n k)$; in particular on the component of the inclusion map $M\to \rc n k$. Now if $e\in\Emb(M, \rc n k)$ is any element, then $e$ defines an automorphism $\varphi\colon M\to e(M)$ and we obtain a commutative diagram
% \[\xymatrix{
%  \Emb(M, \rc n k)  \ar[d] & \Emb(e(M), \rc n k)\ar[l]_{\varphi^*}^\cong \ar[d]\\
%  \Emb(\dell M, \dell \rc n k)  & \Emb(\dell e(M),\dell \rc n k) \ar[l]_{\varphi^*}^\cong
% }\]
% where the horizontal maps are homeomorphisms and $e$ on  the left corresponds to the inclusion map on the right. Therefore the right map is a Serre fibration, hence also the left. 
\end{proof}
% 
% In step 1 we used:
% 
% \begin{lemma}\label{lem:locally_trivial_group_homomorphism}
% A homomorphism $p\colon G\to H$ of topological groups is locally trivial if and only if it has a section in a neighborhood of $e\in H$.
% \end{lemma}
% 
% 
% \begin{proof}
% The ``only if'' statement is obvious. For the converse, suppose that $p$ has a section $s\colon U\to G$ for some neighborhood $U\subset H$ of $e$. Letting $K=\ker(p)$, the homeomorphism
% \[p\inv(U) \to K\times U, \quad g\mapsto (g\cdot sp(g)\inv, p(g))\]
% shows that $p$ is locally trivial over $U$. If $h\in\im(p)$, say, $h=p(g)$, then $hU$ is an open neighborhood of $h$ with preimage $g\cdot p\inv(U)$, hence $p$ is locally trivial over $hU$. 
% 
% It remains to show that the complement of $\im(p)\subset H$ is an open subset. Indeed, if $h\notin \im(p)$, we  claim that $h U\inv \cap \im(p)=\emptyset$. For if there was $g\in G$ such that $p(g) = h\cdot k$, with $k\inv\in U$, then $h= p(g) k\inv\in p(g) U$ is also in $\im(p)$, contradicting the assumption. 
% \end{proof}

%%%%%%%%%%%%%%%%%%%%%%%%%%%% References  %%%%%%%%%%%%%%%%%%%%%%%%%%%%%%%
\typeout{-----------------------  References ------------------------}

\bibliographystyle{abbrv}

%\bibliography{abscobcat}

\end{document}